\documentclass[a4paper]{amsart}
\usepackage[utf8]{inputenc}
\usepackage[english]{babel}

\usepackage{amscd,amssymb}
\usepackage{amsmath}
\usepackage{amsthm}
\usepackage{graphicx}
\usepackage{fancybox}
\usepackage{epic,eepic}
\usepackage{amstext}
\usepackage{mathtools}
\usepackage{esint}
\usepackage[square,numbers]{natbib}
\usepackage{booktabs}
\usepackage[hide links]{hyperref}
\usepackage{comment}
\usepackage{mathtools}
\usepackage{tikz,pgfplots}
\usepackage{subcaption}

\usepackage[noabbrev, capitalise, nameinlink]{cleveref}
\crefname{equation}{}{}
\crefname{assumption}{Assumption}{Assumptions}
\crefformat{equation}{\textup{#2(#1)#3}}

\usepackage{algorithm}
\usepackage{algpseudocode}
\usepackage{tabularx}
\newcolumntype{Y}{>{\centering\arraybackslash}X}

\newtheorem{theorem}{Theorem}[section]

\newtheorem{lemma}[theorem]{Lemma}
\newtheorem{proposition}[theorem]{Proposition}

\theoremstyle{definition}

\theoremstyle{remark}
\newtheorem{remark}[theorem]{Remark}
\numberwithin{theorem}{section}
\numberwithin{equation}{section}
\numberwithin{figure}{section}

\def\Th{\mathcal{T}_h}

\def\with{\,:\,}
\def\dx{\,\mathrm{d}x}
\DeclareMathOperator*{\argmin}{arg\,min}

\numberwithin{equation}{section}
\numberwithin{theorem}{section}

\AtBeginDocument{%
	\def\MR#1{}
}

\begin{document}
\title[A HHO Method for the Gross--Pitaevskii Eigenvalue Problem]{A Hybrid High-Order Method for the Gross--Pitaevskii Eigenvalue Problem}
\author[M.~Hauck, Y.~Liang]{Moritz Hauck$^*$, Yizhou Liang$^\dagger$}
\address{${}^*$ Institute for Applied and Numerical Mathematics, Karlsruhe Institute of Technology, Englerstr.~2, 76131 Karlsruhe, Germany}
\email{moritz.hauck@kit.edu}
\address{${}^{\dagger}$ School for Mathematics, University of Edinburgh, Mayfield Road, James Clerk Maxwell Building, EH9 3FD Edinburgh}
\email{yizhou.liang@ed.ac.uk}	
	
	\begin{abstract}
		We introduce a hybrid high-order method for approximating the ground state of the nonlinear Gross--Pitaevskii eigenvalue problem. Optimal convergence rates are proved for the ground state approximation, as well as for the associated eigenvalue and energy approximations. Unlike classical conforming methods, which inherently provide upper bounds on the ground state energy, the proposed approach gives rise to guaranteed and asymptotically exact lower energy bounds.  Importantly, and in contrast to previous works, they are obtained directly without the need of post-processing, leading to more accurate guaranteed lower energy bounds in practice.
	\end{abstract}
	

	
	\maketitle
	
	\section{Introduction}
	The Gross--Pitaevskii eigenvalue problem (GP-EVP) arises in quantum physics, where it describes stationary quantum states of bosonic particles at ultracold temperatures, known as Bose--Einstein condensates; see, e.g.,~\cite{Dalfovo1999,AboShaeer2001,Fetter2009}. The problem involves a non-negative trapping potential $V \in L^\infty(\Omega)$ and a parameter $\kappa > 0$ describing the strength of the repulsive particle interactions. As computational domain we consider a bounded convex Lipschitz domain $\Omega \subset \mathbb{R}^d$, ${d \in \{2,3\}}$, where we note that the restriction to a sufficiently large domain, along with homogeneous Dirichlet boundary conditions, is a standard and physically reasonable assumption for modeling low-energy quantum states, cf.~\cite{Bao13}. Mathematically, the \mbox{GP-EVP} seeks $L^2$-normalised eigenfunctions $\{u_j \with j =1,2,\dots \} \subset H^1_0(\Omega)$ and corresponding eigenvalues $\{\lambda_j\with j = 1,2,\dots \}$ such that
	\begin{equation}\label{eq:gpe}
		-\Delta u_j + V u_j + \kappa |u_j|^2 u_j = \lambda_j u_j
	\end{equation}
	holds in the weak sense. The function $|u_j|^2$ represents the density of the stationary quantum state $u_j$, and the eigenvalue $\lambda_j$ is typically called the chemical potential. It is a classical result that all eigenvalues of~\cref{eq:gpe} are real and positive, and the smallest eigenvalue is simple; see, e.g.,~\cite{CCM10}. Without loss of generality, we assume that the eigenvalues are ordered nondecreasingly, i.e., $0 < \lambda_1 < \lambda_2 \leq \lambda_3 \leq \dots$.
	
	The GP-EVP arises as the Euler--Lagrange equations for critical points of the Gross--Pitaevskii energy, defined for all $v \in H^1_0(\Omega)$ as,
	\begin{equation}
		\label{eq:energy}
		\mathcal{E}(v) \coloneqq \tfrac{1}{2} (\nabla v, \nabla v)_{L^2} + \tfrac{1}{2} (Vv, v)_{L^2} + \tfrac{\kappa}{4} (|v|^2 v, v)_{L^2},
	\end{equation}
	subject to the $L^2$-normalisation constraint. The ground state, which represents the stationary quantum state of lowest energy, is characterised as the minimiser of the Gross--Pitaevskii energy, i.e., 
	\begin{equation}
		\label{eq:gs}
		u \in \argmin_{v \in H^1_0(\Omega)  \with \|v\|_{L^2} = 1} \mathcal{E}(v).
	\end{equation}
We emphasise that, under the above assumptions on $\Omega$ and $V$, the constrained minimisation problem~\cref{eq:gs} admits a unique global minimiser (up to sign). Moreover, this ground state $u$ coincides, up to sign, with the eigenfunction $u_1$ associated with the smallest eigenvalue $\lambda_1$ of~\cref{eq:gpe}. The minimal energy $E := \mathcal{E}(u)$ is related to the smallest eigenvalue by the identity $ \lambda_1 = 2E + \tfrac{\kappa}{2} \|u\|_{L^4}^4$. For further theoretical results on the Gross--Pitaevskii problem, see, for example,~\cite{CCM10}.
	
A wide range of discretisation methods has been proposed in the literature to approximate the Gross--Pitaevskii ground state. Most existing approaches are based on $H^1_0$-conforming discretisations, including standard continuous finite elements~\cite{Zho04,CCM10,CHZ11}, spectral and pseudospectral methods~\cite{CCM10,Bao13}, multiscale methods~\cite{HMP14b,MR4379968,henning2023optimal,PWZ23}, and mesh-adaptive methods~\cite{Dan10,He21}. Recently, also non-standard conforming finite element methods using mass lumping techniques to preserve certain positivity properties of the ground state have been introduced; see~\cite{positivitygpe,hauck2024positivity}.
A characteristic property of conforming discretisations is that they approximate the ground state energy from above, as the energy functional is minimised over a subset. In the context of linear eigenvalue problems, where the terms energy and eigenvalue can be used interchangeably, several strategies have been developed to compute lower eigenvalue bounds. Among these, mixed finite element methods have proven effective~\cite{Gal23}, and this approach has recently been extended to the present nonlinear GP-EVP in~\cite{GHL24}. However, this method typically requires an additional postprocessing step to obtain asymptotically exact lower bounds.
An alternative approach that gives guaranteed lower eigenvalue bounds without the need for post-processing is provided by  hybrid high-order (HHO) methods; see~\cite{DiPietro2018} for an overview. In the linear setting, such bounds have been established, for example in~\cite{Carstensen2021, Carstensen2024, Tran2024}; see also~\cite{Carstensen2020} for a similar result based on the related hybridisable discontinuous Galerkin method.

In this work, we extend the HHO methodology to the GP-EVP, aiming to obtain guaranteed lower bounds on the ground state energy without the need for post-processing. To this end, we introduce a modified version of the HHO method that employs a lowest-order quadrature rule  for the nonlinearity. This modification is crucial for estimating the nonlinear term in the proof of the guaranteed lower energy bound via Jensen’s inequality. Note that the lack of higher-order generalisations of Jensen’s inequality prevents us from deriving high-order guaranteed lower energy bounds for the Gross–Pitaevskii problem. Although the guaranteed lower energy bounds we obtain are only of lowest order, they can still be significantly more accurate than those from the post-processed lowest-order mixed discretization in~\cite{GHL24}. The improved accuracy is due to the absence of post-processing, which can  degrade the quality of the bounds, particularly for smooth problems where discretization errors are comparatively small. Numerical experiments clearly demonstrate this improvement, showing that the guaranteed lower energy bounds of the proposed modified HHO method are more accurate by up to two orders of magnitude than the post-processed mixed discretisation from~\cite{GHL24}.

In addition to providing guaranteed lower energy bounds, HHO methods offer improved convergence rates for the reconstructed unknowns compared to classical finite element methods. Specifically, when using \(k\)-th order polynomials on the mesh faces and assuming sufficient smoothness, the reconstructed ground state approximation exhibits optimal convergence rates of order~\(k+1\) in the \(H^1\)-seminorm and \(k+2\) in the \(L^2\)-norm. We provide a rigorous convergence analysis establishing these rates, along with optimal convergence of order \(2k+2\) for the energy and eigenvalue approximations. 
Furthermore, we analyse the convergence of the proposed modified HHO method and prove first-order $H^1$-convergence of the ground state approximation, second-order $L^2$-convergence, and second-order convergence for the energy and eigenvalue approximations.
Notably, the presented error analysis differs substantially from standard techniques in the linear setting. Instead, it is inspired by the analysis conducted in~\cite{CCM10, GHL24} for classical conforming and mixed discretisations of the Gross--Pitaevskii problem, respectively.
	
The remainder of the paper is organised as follows. \Cref{sec:hho} introduces the HHO method for the Gross--Pitaevskii problem. Guaranteed lower energy bounds for a modified HHO method are established in \cref{sec:gleb}. \Cref{sec:convergence} presents an a priori convergence analysis of the HHO method and its modified version. Numerical experiments supporting our theoretical results are presented in \cref{sec:numexp}.

	\section{Hybrid high-order method}\label{sec:hho}
Consider a hierarchy of simplicial meshes $\{\mathcal{T}_h\}_{h > 0}$ of the domain $\Omega$, which we assume to be geometrically conforming and shape regular.  We denote by $T$ the elements of  $\mathcal{T}_h$, and define the maximal mesh size of $\Th$ as $
h \coloneqq \max_{T \in \mathcal{T}_h} h_T$, where $h_T \coloneqq \operatorname{diam}(T)$.
The set of mesh faces of \( \mathcal{T}_h \), denoted by \( \mathcal{F}_h \), is partitioned into the set of interior faces \( \mathcal{F}_h^i \) and boundary faces \( \mathcal{F}_h^b \).
 For any \( T \in \mathcal{T}_h \), we denote by \( \mathcal{F}_{\partial T} \) the set of faces lying on the boundary of \( T \).
The space of polynomials of total degree at most \( l \in \mathbb{N}_0 \) on an element \( T \in \mathcal{T}_h \) is denoted by \( \mathcal P^l(T) \). Similarly, \( \mathcal P^k(F) \) denotes the space of polynomials of total degree at most \( k \in \mathbb{N}_0 \) on a face \( F \in \mathcal{F}_h \). 
Moreover, for any $T \in \mathcal T_h$, the space of broken polynomials on $\partial T$ is defined as
\[
\mathcal P^k(\mathcal{F}_{\partial T}) := \left\{ v \in L^2(\partial T) \with v|_F \in \mathcal P^k(F), \ \forall F \in \mathcal{F}_{\partial T} \right\}.
\]
We further introduce for all \( T \in \mathcal{T}_h \) and all \( F \in \mathcal{F}_h \) the  projections
$
\Pi_T^l \colon L^2(T) \to \mathcal{P}^l(T)$ and $\Pi_{F}^k \colon L^2(F) \to \mathcal{P}^k(F),
$
which are defined as orthogonal projections with respect to the \( L^2 \)-inner products \( (\cdot, \cdot)_T \) and \( (\cdot, \cdot)_{F} \), respectively.

Global discrete polynomial spaces can be obtained by concatenating the local spaces \( \mathcal P^l(T) \) and~\( \mathcal P^k(F) \) in a discontinuous manner, which gives
\begin{align*}
	\mathcal P^l(\mathcal{T}_h)&\coloneqq \{v \in L^2(\Omega) \with v|_T \in \mathcal P^l(T),\; \forall T \in \Th \},\\
	\mathcal P^k(\mathcal{F}_h)&\coloneqq \{v \in L^2(\Sigma) \with v|_F \in \mathcal P^k(F),\;  \forall F \in \mathcal F_h \},
\end{align*}
where $\Sigma \coloneqq \cup_{F \in \mathcal F_h} F$ denotes the mesh skeleton. 
Corresponding \( L^2 \)-projections can be defined piecewise as follows: \( \Pi_{\mathcal{T}_h}^l \colon L^2(\Omega) \to \mathcal{P}^l(\mathcal{T}_h) \) by \( (\Pi_{\mathcal{T}_h}^l \cdot)|_T \coloneqq \Pi_T^l \cdot \) for all \( T \in \mathcal{T}_h \), and \( \Pi_{\mathcal{F}_h}^k \colon L^2(\Sigma) \to \mathcal{P}^k(\mathcal{F}_h) \) by \( (\Pi_{\mathcal{F}_h}^k \cdot)|_F \coloneqq \Pi_F^k \cdot \) for all \( F \in \mathcal{F}_h \).




The global approximation space of the HHO method is given by
\begin{equation}
	\label{eq:appsp}
 \hat{V}_h \coloneqq \mathcal P^{k+1}(\mathcal{T}_h) \times \mathcal P^k(\mathcal{F}_h),
\end{equation}
where we choose $l = k+1$, which is a classical choice for HHO methods in the context of guaranteed lower eigenvalue bounds. Note that the hat notation emphasises the presence of both element and face components. For a function \( \hat{v}_h \in \hat{V}_h \), these components are  denoted by \( \hat{v}_h = (v_{\mathcal{T}_h}, v_{\mathcal{F}_h}) \), with \( v_{\mathcal{T}_h} = (v_T)_{T \in \mathcal{T}_h} \) and \( v_{\mathcal{F}_h} = (v_F)_{F \in \mathcal{F}_h} \). Furthermore, for any element $T \in \Th$, we denote the corresponding  restriction of a function \( \hat{v}_h \in \hat V_h\) by \( \hat{v}_T = (v_T, v_{\partial T}) \in \hat V_T \coloneqq \mathcal P^{k+1}(T) \times \mathcal P^{k}(\mathcal F_{\partial T})\), where~\( v_{\partial T}\) is given, for any \( F \in \mathcal{F}_{\partial T} \), as \( v_{\partial T}|_F = v_F \). The subspace of \( \hat{V}_h \) incorporating homogeneous Dirichlet boundary conditions is  defined as
\begin{equation}
	\label{eq:appsp0}
	\hat U_h \coloneqq \{ \hat v_h \in \hat{V}_h \with v_F = 0,\; \forall F  \in \mathcal{F}_h^b\}.
\end{equation}

A central component of the HHO methodology is the  reconstruction operator $\mathcal R_h\colon \hat V_h \to \mathcal P^{k+1}(\mathcal T_h)$. Given  $\hat v_h \in \hat V_h$, its reconstruction $\mathcal R_h\hat v_h \in \mathcal P^{k+1}(\mathcal T_h)$ is defined  as the unique function satisfying, for all $T \in \mathcal T_h$ and \( \varphi \in \mathcal P^{k+1}(T) \),
\begin{subequations}
	\label{eq:reconstruction}
	\begin{align}
		(\nabla (\mathcal R_h \hat{v}_h)|_T, \nabla \varphi)_{L^2(T)} &= -(v_{T}, \Delta \varphi)_{L^2(T)} + (v_{\partial T}, \partial_n \varphi)_{L^2(\partial T)},\label{eq:local_restruct}\\
		((\mathcal R_h \hat{v}_h)|_T, 1)_{L^2(T)} &= (v_{T}, 1)_{L^2(T)},\label{eq:local_restructb}
	\end{align}
\end{subequations}
where \( \partial_n \) denotes the outward normal derivative on \( \partial T \). 

Having introduced the reconstruction operator, we now introduce the discrete bilinear form of the HHO method. For a given stabilisation parameter \( \sigma > 0 \), it is defined for all \( \hat{v}_h, \hat{\varphi}_h \in \hat{V}_h \) by
\begin{align*}
	a_h(\hat{v}_h, \hat{\varphi}_h) \coloneqq  (\nabla_h \mathcal R_h \hat{v}_h, \nabla_h \mathcal R_h \hat{\varphi}_h)_{L^2} +   s_h(\hat v_h,\hat \varphi_h),
\end{align*}
where $\nabla_h$ denotes the $\mathcal T_h$-piecewise gradient. The stabilisation bilinear form \(s_h\) is given by \( s_h(\hat{v}_h, \hat{\varphi}_h) \coloneqq \sum_{T \in \mathcal{T}_h} s_T(\hat{v}_T, \hat{\varphi}_T) \) with \( s_T \) for all \( \hat{v}_T, \hat{\varphi}_T \in \hat{V}_T \) defined~as
\begin{multline*}
	s_T(\hat{v}_T, \hat{\varphi}_T) \coloneqq  \sigma\hspace{-.5ex}\sum_{F \in \mathcal{F}_{\partial T}}\hspace{-1ex}\Big\{ \ell_{T,F}^{-1} \left( \Pi_F^k \left( v_{\partial T} - (\mathcal{R}_h \hat{v}_h)_T \right),\, \Pi_F^k \left( \varphi_{\partial T} - (\mathcal{R}_h \hat{\varphi}_h)_T \right) \right)_{L^2(F)}\hspace{-.5ex}\Big\}\\[-1ex]
+	\sigma h_T^{-2} \left( v_T - (\mathcal{R}_h \hat{v}_h)_T,\, \varphi_T - (\mathcal{R}_h \hat{\varphi}_h)_T \right)_{L^2(T)}.
\end{multline*}
Here, the weights \( \ell_{T,F} > 0 \) are for any $T \in \mathcal T_h$ and $F \in \mathcal{F}_{\partial T}$ defined as\[
\ell_{T,F} \coloneqq \frac{|F|_{d-1} \, h_T^2}{|T_F|_d},
\]
where \(T_F \coloneqq \operatorname{conv}\{x_T, F\} \subset T\) is the convex hull of $x_T$, denoting the barycenter of the element $T$, and the face \(F\). Further, \(|S|_n\) for \(n \in \{1, \dots, d\}\) denotes the \(n\)-dimensional volume of the set \(S\).
Due to shape regularity of \(\mathcal{T}_h\), we have the scaling \(\ell_{T,F} \approx h_T\).
The bilinear form \(s_h\) serves to penalise the non-conformity of the discrete solution \(\hat{u}_h = (u_{\mathcal{T}_h}, u_{\mathcal{F}_h})\), in particular the discrepancy between \(u_{\mathcal{T}_h}\) and \(u_{\mathcal{F}_h}\) across the mesh skeleton \(\Sigma\).
Note that the above choice of stabilisation was introduced in the context of guaranteed lower eigenvalue bounds in~\cite[Eq.~(3.6)]{Tran2024}, and a related stabilisation for the case \(l = k\) was proposed in~\cite[Ex.~2.8]{Pietro2020}. The latter stabilisation, however, is unstable for our choice \(l = k+1\), cf.~\cite{Tran2024}. The discrete bilinear forms \(a_h\) and \(s_h\) induce corresponding discrete (semi-)norms, defined by \(\| \cdot \|_{a_h}^2 \coloneqq a_h(\cdot, \cdot)\) and \(| \cdot |_{s_h}^2 \coloneqq s_h(\cdot, \cdot)\), respectively.

 A discrete counterpart of the energy~\cref{eq:energy} can be defined, for any $\hat v_h \in \hat U_h$, as
\begin{equation*}
	\mathcal{E}_h(\hat{v}_h) \coloneqq  \tfrac{1}{2} a_h(\hat{v}_h, \hat{v}_h) + \tfrac{1}{2}(V v_{\mathcal{T}_h}, v_{\mathcal{T}_h})_{L^2} + \tfrac{\kappa}{4}(|v_{\mathcal{T}_h}|^2 v_{\mathcal{T}_h}, v_{\mathcal{T}_h})_{L^2},
\end{equation*}
and the resulting HHO approximation \( \hat{u}_h \in \hat U_h \) to the ground state \( u \) is characterised as the solution to the following discrete constrained minimisation problem:
\begin{equation}
	\label{eq:hhoapprox}
	\hat{u}_h \in \argmin_{\hat{v}_h \in \hat U_h \with \|v_{\mathcal{T}_h}\|_{L^2} = 1} \mathcal{E}_h(\hat{v}_h),
\end{equation}
where \( E_h \coloneqq \mathcal{E}_h(\hat{u}_h) \) denotes the corresponding discrete ground state energy. 
Unlike the continuous case, where the ground state is unique up to sign, uniqueness properties in the discrete setting remain an open question, with  results only known for the first-order lumped continuous finite element method, cf.~\cite{positivitygpe,hauck2024positivity}. Nevertheless, this non-uniqueness is typically not problematic in practice, and we proceed under the assumption that a discrete minimiser has been found. Such a discrete minimiser always exists by finite-dimensional compactness arguments (Bolzano-Weierstrass theorem). To align the sign of the discrete solution with that of the continuous ground state \( u \), we choose the sign of \( \hat u_h \) such that \( (u, u_{\mathcal T_h})_{L^2} \geq 0 \).

\section{Guaranteed lower energy bound}\label{sec:gleb}
Due to the nonlinearity, classical techniques for obtaining guaranteed lower eigenvalue bounds with HHO methods are not directly applicable. The approach in~\cite{GHL24}, which establishes such bounds for the Gross--Pitaevskii problem using lowest-order mixed finite elements and a post-processing step, relies on Jensen's inequality to handle the nonlinearity. This is feasible because the ground state is approximated by \(\mathcal{T}_h\)-piecewise constants and Jensen's inequality is applicable in this setting. However, for the HHO method defined in \cref{eq:hhoapprox}, it is unclear how to apply Jensen's inequality, since the bulk approximation space consists at least of \(\mathcal{T}_h\)-piecewise first-order polynomials. Corresponding generalizations of Jensen's inequality are not known (and also not expected to hold).

To still obtain a guaranteed lower energy bound using the HHO methodology, we introduce a modified discrete energy functional, defined for all \( \hat{v}_h \in \hat{U}_h \) as
\begin{equation}\label{eq:energymod}
	\mathcal{E}_h^0(\hat{v}_h)\coloneqq \tfrac{1}{2}a_h(\hat{v}_h,\hat{v}_h) + \tfrac{1}{2}(V v_{\mathcal T_h},v_{\mathcal T_h})_{L^2} + \tfrac{\kappa}{4}(|\Pi_{\mathcal T_h}^0 v_{\mathcal T_h}|^2 v_{\mathcal T_h},v_{\mathcal T_h})_{L^2}.
\end{equation}
The special treatment of the nonlinearity can be interpreted as applying a low-order quadrature rule to (parts of) the nonlinear term. The corresponding modified HHO approximation \( \hat{u}_h^0 \in \hat U_h \) is then defined as the solution to:
\begin{equation}
	\label{eq:hhoapprox0}
	\hat{u}_h^0\in\argmin_{\hat{v}_h\in\hat U_h\with \|v_{\mathcal T_h}\|_{L^2}=1} \mathcal E_h^0 (\hat{v}_h).
\end{equation}
Note that the modified HHO approximation is considered only for \( k = 0 \), as the low-order quadrature in \cref{eq:hhoapprox0} prevents any gain in convergence rates for \( k > 0 \).

To derive guaranteed lower energy bounds, one typically exploits the minimising property of~\cref{eq:hhoapprox0} by bounding the discrete energy from above through evaluating the discrete energy functional at a suitable interpolation of the exact solution. For this purpose, we introduce an interpolation operator as
\begin{equation*}
	\mathcal I_h\colon H^1_0(\Omega) \to \hat U_h, \qquad \mathcal I_h v \coloneqq \big( \Pi^{k+1}_{\mathcal T_h} v,\, \Pi^k_{\mathcal F_h} v \big).
\end{equation*}
The following important operator identity holds:
\begin{equation}
	\label{eq:keyidentity}
	\mathcal R_h\circ \mathcal I_h = \mathcal G_h,
\end{equation}
where  \( \mathcal G_h \colon H^1_0(\Omega) \to \mathcal P^{k+1}(\mathcal T_h) \) denotes the elliptic projection. Given a function $v \in H^1_0(\Omega)$, its elliptic projection $\mathcal G_h v \in \mathcal P^{k+1}(\mathcal T_h)$ is defined as the unique function satisfying, for all $T \in \mathcal T_h$ and \( w \in \mathcal P^{k+1}(T) \), the two conditions
\begin{subequations}
	\begin{align}
		(\nabla (\mathcal G_h v)|_T, \nabla w)_{L^2(T)} &= (\nabla v, \nabla w)_{L^2(T)}, \\
		((\mathcal G_h v)|_T, 1)_{L^2(T)} &= (v, 1)_{L^2(T)}.
	\end{align}
\end{subequations}

The desired guaranteed lower energy bound for the modified HHO approximation is stated in the following theorem.

\begin{theorem}[Guaranteed lower energy bound]\label{theo:lower_bound}
	Assume that $V \in \mathcal P^0(\mathcal T_h)$ and let the stabilisation parameter $\sigma$ and the mesh size $h$ be chosen such that
	\begin{equation}
		\label{eq:condstabilization}
		1- \sigma(\tfrac{1}{\pi^2}+C_{\mathrm{tr}}) - \tfrac{4h^2E_h^0}{\pi^2}\geq 0
	\end{equation}
	with the constant $C_{\mathrm{tr}} \coloneqq 1/\pi^2+2/(d\pi)>0$. Then, there holds the following guaranteed lower energy bound:
	\begin{equation*}
		E_h^0\leq E.
	\end{equation*}
\end{theorem}
\begin{proof}
	The discrete energy of the ground state approximation $\hat u_h^0$ from \cref{eq:hhoapprox0} is characterised by the pseudo-Rayleigh quotient
	\begin{multline*}
		E_h^0 = \min_{\hat{v}_h\in \hat U_h \with \|v_{\mathcal T_h}\|_{L^2}>0 }\frac{1}{{\|v_{\mathcal T_h}\|_{L^2}^4}}\Big\{\tfrac{1}{2}(\nabla_h \mathcal R_h \hat{v}_h,\nabla_h \mathcal R_h \hat{v}_h)_{L^2}\|v_{\mathcal T_h}\|_{L^2}^2		\\
		\qquad \qquad \;+ \tfrac{1}{2}s_h(\hat{v}_h,\hat{v}_h)\|v_{\mathcal T_h}\|_{L^2}^2
		+\tfrac{1}{2}(Vv_{\mathcal T_h},v_{\mathcal T_h})_{L^2}\|v_{\mathcal T_h}\|_{L^2}^2 + \tfrac{\kappa}{4}(|\Pi_{\mathcal T_h}^0 v_{\mathcal T_h}|^2v_{\mathcal T_h},v_{\mathcal T_h})_{L^2}\Big\}.
	\end{multline*}
	Next, we majorize $E_h^0$ by choosing $\hat{v}_h = \mathcal I_hu \in \hat U_h$ and use \cref{eq:keyidentity}, which results in
	\begin{align*}
			E_h^0 \|\Pi_{\mathcal T_h}^{k+1}u\|_{L^2}^4 &\leq \tfrac{1}{2}(\nabla_h \mathcal G_hu,\nabla_h \mathcal G_hu)_{L^2}\|\Pi_{\mathcal T_h}^{k+1}u\|_{L^2}^2+ \tfrac{1}{2}s_h(\mathcal I_hu,\mathcal I_hu)\|\Pi_{\mathcal T_h}^{k+1}u\|_{L^2}^2\notag\\[.5ex]
			&\quad+\tfrac{1}{2}(V \Pi_{\mathcal T_h}^{k+1}u,\Pi_{\mathcal T_h}^{k+1}u)_{L^2}\|\Pi_{\mathcal T_h}^{k+1}u\|_{L^2}^2 
			+ (|\Pi_{\mathcal T_h}^{0}u|^2\Pi_{\mathcal T_h}^{k+1}u,\Pi_{\mathcal T_h}^{k+1}u)_{L^2}.\label{eq:lhs}.
	\end{align*}
	
	In the following, we bound all the terms on the right-hand side of the latter inequality individually, where we note that $\|\Pi_{\mathcal T_h}^{k+1}u\|_{L^2}^2\leq \|u\|_{L^2}^2 =1$. For the first term, we obtain using Pythagoras theorem, that
\begin{align*}
	 (\nabla_h\mathcal G_hu,\nabla_h\mathcal G_hu)_{L^2}
	 = \|\nabla u\|_{L^2}^2 - \|\nabla_h(u-\mathcal G_hu)\|_{L^2}^2.
\end{align*}
The second term can be estimated using \cref{eq:keyidentity},  \cref{lem:poincare,lem:trace_theo}, and the stability properties of the $L^2$-projections $\Pi_{\mathcal T_h}^{k+1}$ and $\Pi_{\mathcal F_h}^k$, as
\begin{align*}
	s_h(\mathcal I_hu,\mathcal I_hu) 
	\leq \sigma(\tfrac{1}{\pi^2}+C_{\text{tr}})\|\nabla_h(u-\mathcal G_hu)\|_{L^2}^2.
\end{align*}
For estimating the third term, let us recall that $V \in \mathcal P^0(\mathcal T_h)$. This allows to split up the $L^2$-inner product into local element contributions. Using the local stability properties of $\Pi_T^{k+1}$, this yields that
\begin{align*}
	(V \Pi_{\mathcal T_h}^{k+1}u,\Pi_{\mathcal T_h}^{k+1}u)_{L^2} = \sum_{T \in \mathcal{T}_h} V|_T \|\Pi_T^{k+1}u\|_{L^2(T)}^2\leq  (Vu,u)_{L^2}.
\end{align*}
To estimate the fourth term, we again split up the inner product in local element contributions. Noting that $\Pi_{T}^0  \Pi_{T}^{k+1}v = \Pi_{T}^0v$ for all $v \in L^2(T)$, we obtain with Jensen's inequality for all $T \in \mathcal T_h$ that
	\begin{align*}
		(|\Pi_T^0u|^2\Pi_T^{k+1} u,\Pi_T^{k+1} u)_T\leq  \Big(\fint_T u\dx\Big)^2 (u,u)_T \leq |T|_d \Big(\fint_T u^2\dx\Big)^2  \leq 
		 (|u|^2u,u)_T.
	\end{align*}
	Summing up over all $T \in \mathcal T_h$ then results in 
	\begin{equation*}
		(|\Pi_{\mathcal T_h}^0u|^2\Pi_{\mathcal T_h}^{k+1} u,\Pi_{\mathcal T_h}^{k+1} u)_{L^2}\leq (|u|^2u,u)_{L^2}.
	\end{equation*}
	
	Combining the previous estimates gives 
		\begin{align}\label{eq:lhs}
		E_h^0 \|\Pi_{\mathcal T_h}^{k+1} u\|_{L^2}^4 \leq E - \tfrac12\big(1-\sigma(\tfrac1{\pi^2} + C_\mathrm{tr})\big) \|\nabla_h(u-\mathcal G_h u)\|_{L^2}^2.
	\end{align}
		For rewriting the left-hand side of \cref{eq:lhs}, we use the Pythagorean identity and the best-approximation property of the $L^2$-projection $\Pi_{\mathcal T_h}^{k+1}$ and  \cref{lem:poincare} to get that
		\begin{align*}
				\|\Pi_{\mathcal T_h}^{k+1}u\|_{L^2}^4 & = \big(\|u\|_{L^2}^2-\|u-\Pi_{\mathcal T_h}^{k+1}u\|_{L^2}^2\big)^2
				\geq \big(1-\|u-\mathcal G_hu\|_{L^2}^2\big)^2
				\\
				& \geq 1-\tfrac{2h^2}{\pi^2}\|\nabla_h(u-\mathcal G_hu)\|_{L^2}^2.
			\end{align*}
		Finally, the combination of the previous estimates leads to the inequality
		\begin{equation*}
				E_h^0\leq E - \tfrac{1}{2}\big(1- \sigma(\tfrac{1}{\pi^2}+C_{\mathrm{tr}}) - \tfrac{4h^2E_h^0}{\pi^2}\big)\|\nabla_h(u-\mathcal G_hu)\|_{L^2}^2,
			\end{equation*}
		which completes the proof.
\end{proof}

\begin{remark}[Assumption of \(\mathcal{T}_h\)-piecewise constant potentials]
The assumption \(V \in \mathcal{P}^0(\mathcal{T}_h)\) in \cref{theo:lower_bound} is essential for proving the guaranteed lower energy bound. For simple potentials, where the minimum value on each element can be easily computed (e.g., the harmonic potential \(V(x) = \tfrac{1}{2} |x|^2\)), one can approximate \(V\) by a piecewise constant function that assigns this minimum value to each element. Applying the modified HHO method \cref{eq:hhoapprox0} to this piecewise constant potential approximation then yields a guaranteed lower energy bound for the original~problem.
\end{remark}

\section{A priori error analysis}\label{sec:convergence}

In this section, we present an a priori error analysis of the HHO approximation introduced in \cref{eq:hhoapprox} and its modified version from \cref{eq:hhoapprox0}. The analysis is inspired by the works~\cite{Zho04,CCM10} on classical conforming finite element methods. Results from the conforming setting are used at several points in the analysis, along with technical tools that enable their application in the present non-conforming setting.
	
The Euler--Lagrange equations associated with the constrained minimisation problem \cref{eq:gs} give rise to the following eigenvalue problem: seek \( (u, \lambda) \in H^1_0(\Omega) \times \mathbb{R} \) with \( \|u\|_{L^2(\Omega)} = 1 \) such that, for any $\varphi \in H^1_0(\Omega)$, it holds that
\begin{equation}\label{eq:eigen_pro:conti}
	(\nabla u, \nabla \varphi)_{L^2} + (V u, \varphi)_{L^2} + \kappa(u^3, \varphi)_{L^2} = \lambda(u, \varphi)_{L^2}.
\end{equation}
Recall that the eigenvalue  associated with the ground state, referred to as the ground state eigenvalue, is the smallest among all eigenvalues of the problem.

Similarly, also any discrete ground state defined as the solution to the discrete constrained minimisation problem \cref{eq:hhoapprox} satisfies the discrete eigenvalue problem: seek $(u_h,\lambda_h) \in \hat U_h \times \mathbb R$ with $\|u_{\mathcal T_h}\|_{L^2} = 1$ such that, for all $\hat \varphi_h \in \hat U_h$, it holds that
\begin{equation}\label{eq:eigen_pro:discrete}
	a_h(\hat{u}_h,\hat{\varphi}_h) + (Vu_{\mathcal{T}_h},\varphi_{\mathcal{T}_h})_{L^2} +\kappa  (u_{\mathcal{T}_h}^3,\varphi_{\mathcal{T}_h})_{L^2}  = \lambda_h (u_{\mathcal{T}_h},\varphi_{\mathcal{T}_h})_{L^2
	},
\end{equation}
where we recall the notation \( \hat{u}_h = (u_{\mathcal{T}_h}, u_{\mathcal{F}_h}) \) and analogously for the test function~\( \hat{\varphi}_h \).
The discrete ground state eigenvalue \( \lambda_h \) is not necessarily the smallest among all discrete eigenvalues. As in the continuous setting, the discrete ground state energy and corresponding eigenvalue are related by $\lambda_h = 2E_h +\tfrac{\kappa}{2}\|u_{\mathcal T_h}\|_{L^4}^4$.

\begin{remark}[Tilde notation]
	In the following, we will write $a \lesssim b$ or $b\gtrsim a$ if it holds that $a \leq C b$ or $a \geq C b$, respectively, where $C>0$ is a constant that may depend on the domain, the mesh regularity, the coefficients $V$ and $\kappa$, the ground state $u$, and the polynomial degree $k$, but is independent of the mesh size $h$.
\end{remark}

Our first objective is to prove the plain convergence of the HHO method; precise convergence rates will be derived later.

\begin{theorem}[Plain convergence of HHO method]\label{theo:plain_conver}
	As $h\rightarrow 0$, it holds that
	\begin{equation*}
		 \|\nabla_h(\mathcal R_h\hat{u}_h -u)\|_{L^2} \rightarrow 0,\quad \|u-u_{\mathcal{T}_h}\|_{L^2}\rightarrow 0,\quad 
		 E_h\rightarrow E,\quad \lambda_h\rightarrow \lambda.
	\end{equation*}
\end{theorem}
\begin{proof}
We begin by establishing the uniform boundedness of the discrete ground state energies~\( E_h \). Using definition \cref{eq:hhoapprox}, we directly obtain the estimate \( E_h \leq \mathcal{E}_h(\mathcal{I}_h u / \|\Pi_{\mathcal{T}_h}^{k+1} u\|_{L^2}) \). Its right-hand side can be bounded independently of \( h \) by an argument similar to that used in the proof of \cref{theo:lower_bound}. Specifically, we apply \cref{lem:trace_theo}, \cref{eq:keyidentity}, along with the uniform boundedness of \( \mathcal{G}_h \) in the \( H^1 \)-seminorm and of \( \Pi_{\mathcal{T}_h}^{k+1} \) in the \( L^4 \)-norm. Importantly, this bound is obtained without requiring any restriction on the stabilization parameter \( \sigma \).
As a direct consequence, we obtain the uniform boundedness of \( \lambda_h \) and \( \|\hat{u}_h\|_{a_h} \), which in turn implies the uniform \( L^6 \)-boundedness of \( u_{\mathcal{T}_h} \) using \cref{prop:discrete:sobolev:embed}.

Consider the auxiliary Poisson problem, which seeks \( u_h^c \in H_0^1(\Omega) \) such that
\begin{equation}
	\label{eq:ustar}
	-\Delta u_h^c = -V u_{\mathcal{T}_h} - \kappa u_{\mathcal{T}_h}^3 + \lambda_h u_{\mathcal{T}_h} \eqqcolon f_h
\end{equation}
holds in the weak sense. Since the \( L^2 \)-norm of the right-hand side \( f_h \) is uniformly bounded, classical elliptic regularity results imply that \( u_h^c \in H^2(\Omega) \cap H_0^1(\Omega) \), with \( H^2 \)-norm uniformly bounded. Noting that \( \hat{u}_h \) is the HHO approximation of the solution~\( u_h^c \) to Poisson problem \cref{eq:ustar} with  right-hand side \( f_h \), classical approximation results for the HHO method (see, e.g., \cite[Thm.~2.27~\&~2.28]{Pietro2020}) show that 
\begin{align}
	\label{eq:errorloworder}
	\|\hat{u}_h - \mathcal I_hu_h^c\|_{a_h} +
	\|\nabla_h(\mathcal R_h\hat{u}_h -u_h^c)\|_{L^2} + |\hat{u}_h|_{s_h}\lesssim h.
\end{align}
From the latter estimate we can also derive the $L^2$-error estimate
\begin{align}
		\label{eq:errorloworderL2}
	\|u_{\mathcal{T}_h}-u_h^c\|_{L^2}&\leq \|u_{\mathcal{T}_h}- \Pi_{\mathcal T_h}^{k+1}u_h^c\|_{L^2} + \|u_h^c- \Pi_{\mathcal T_h}^{k+1}u_h^c\|_{L^2}\lesssim h,
\end{align}
where we used \cref{prop:discrete:sobolev:embed} and the approximation properties of $\Pi^{k+1}_{\mathcal T_h}$.

Next, we define an \( L^2 \)-normalised version of \( u_h^c \) as \( \tilde{u}_h^c \coloneqq u_h^c / \|u_h^c\|_{L^2} \). Using the normalisation condition \( \|u_{\mathcal{T}_h}\|_{L^2} = 1 \) and \cref{eq:errorloworderL2}, one can show that \({ |\|u_h^c\|_{L^2} - 1| \lesssim h} \), so that the normalisation introduces only a perturbation of order~\( h \). Therefore,  we can derive from  \cref{eq:errorloworder} and \cref{eq:errorloworderL2} the estimate
\begin{equation}
	\label{eq:errornormalizedstar}
	  \|\nabla_h(\mathcal R_h\hat{u}_h -\tilde{u}_h^c)\|_{L^2} + \|u_{\mathcal{T}_h} - \tilde{u}_h^c\|_{L^2}\lesssim h.
\end{equation}
For the difference between \( E_h \) and \( E_h^c \coloneqq \mathcal{E}(\tilde{u}_h^c) \) we obtain that
\begin{align}
	\label{eq:estenergiesstar}
	\begin{split}
		&	|E_h - E_h^c| \leq \tfrac{1}{2} \left| (\nabla_h \mathcal{R}_h \hat{u}_h, \nabla_h \mathcal{R}_h \hat{u}_h)_{L^2} - (\nabla \tilde{u}_h^c, \nabla \tilde{u}_h^c)_{L^2} \right| + \tfrac{1}{2} \left| s_h(\hat{u}_h, \hat{u}_h) \right| \\
		&\; 	+ \tfrac{1}{2} \left| (V u_{\mathcal{T}_h}, u_{\mathcal{T}_h})_{L^2} - (V \tilde{u}_h^c, \tilde{u}_h^c)_{L^2} \right| + \tfrac{\kappa}{4} \left| (u_{\mathcal{T}_h}^3, u_{\mathcal{T}_h})_{L^2} - ((\tilde{u}_h^c)^3, \tilde{u}_h^c)_{L^2} \right| \lesssim h,
	\end{split}
\end{align}
where the last estimate follows from \cref{eq:errornormalizedstar,eq:errorloworder}, the uniform \( L^6 \)-boundedness of \( u_{\mathcal{T}_h} \), and the uniform \( L^\infty \)-boundedness of \( \tilde{u}_h^c \). The latter is a consequence of the Sobolev embedding \( H^2(\Omega) \hookrightarrow \mathcal{C}^0(\overline{\Omega}) \) and the uniform $H^2$-boundedness of \( \tilde{u}_h^c \).

Noting that \( E_h - E \leq \mathcal{E}_h\left( \mathcal{I}_h u / \|u_{\mathcal{T}_h}\|_{L^2} \right) - E \), and proceeding similar as in the first part of this proof where we establish the uniform boundedness of \( E_h \), we obtain that \( E_h - E \lesssim h^2 \). Together with \( E \leq E_h^c \) and \cref{eq:estenergiesstar}, this leads to
\[
-h \lesssim E_h - E_h^c \leq E_h - E \lesssim h^2,
\]
which directly implies that \( E_h \to E \) as \( h \to 0 \).
Using an argument similar to that in \cite[Thm.~1]{CCM10}, one can further show that \( \|u - \tilde{u}_h^c\|_{H^1} \to 0 \). In combination with estimate~\cref{eq:errornormalizedstar}, this implies the plain convergence results
\begin{equation}
	\label{eq:L2estimateu}
	\|u - u_{\mathcal{T}_h}\|_{L^2} \to 0, \quad \|\nabla_h(\mathcal{R}_h \hat{u}_h - u)\|_{L^2} \to 0.
\end{equation}
From the first estimate in \cref{eq:L2estimateu}, we additionally obtain that
\[
|\lambda - \lambda_h| \leq 2|E - E_h| + \tfrac{\kappa}{2} \left| \|u_{\mathcal{T}_h}\|_{L^4}^4 - \|u\|_{L^4}^4 \right| \to 0,
\]
using the uniform \( L^6 \)-boundedness of \( u_{\mathcal{T}_h} \), along with an \( L^\infty \)-estimate for \( u \). \qedhere
\end{proof}

An important step in proving convergence rates is to write the ground state \( u \in H^1_0(\Omega) \) as the weak solution to the auxiliary Poisson problem
\begin{equation}
	\label{eq:auxpoisson}
	- \Delta u = \lambda u - V u - \kappa u^3 \eqqcolon f,
\end{equation}
with homogeneous Dirichlet boundary conditions on $\partial \Omega$. This allows us to construct a discrete ground state approximation by considering the HHO discretisation of problem~\cref{eq:auxpoisson}. The latter seeks \( \hat{v}_h =(v_{\mathcal T_h},v_{\mathcal F_h})\in \hat{U}_h \) such that, for all \( \hat{\varphi}_h \in \hat{U}_h \),
\begin{equation}
	a_h(\hat{v}_h, \hat{\varphi}_h) = (f, \varphi_{\mathcal{T}_h})_{L^2},
\end{equation}
and classical HHO theory shows optimal convergence rates for this approximation; see, e.g.,~\cite[Thm.~2.27~\&~2.28]{Pietro2020}. We emphasise that this approximation is introduced solely for theoretical analysis and is not computed in practice. Since \( \hat{v}_h \) is generally not \( L^2 \)-normalised, we define its $L^2$-normalised counterpart by
\[
\hat{w}_h \coloneqq {\hat{v}_h}\big/{\|v_{\mathcal{T}_h}\|_{L^2}}.
\]
In the following convergence proof, we apply the triangle inequality to split the error into two components: the error between \( u \) and \( \hat{w}_h \), and the error between \( \hat{w}_h \) and \( \hat{u}_h \). The first of these two errors is estimated in the lemma below. In what follows, we denote for any $s \in \mathbb N$ by $H^s(\mathcal T_h)$ the broken Sobolev space consisting of functions whose restriction to each element $T \in \mathcal T_h$ belongs to $H^s(T)$.  

\begin{lemma}[Bound of first term]\label{lem:firstterm}
	Assume that, for some $0\leq r\leq k$, it holds that $u\in H^{r+2}(\mathcal{T}_h)$. Then, we have the following approximation results:
	\begin{equation}\label{eq:err1}
		\|w_{\mathcal T_h} - u\|_{L^2} + \|\hat{w}_h - \mathcal I_hu\|_{a_h} +	\|\nabla_h(\mathcal R_h\hat{w}_h -u)\|_{L^2} + |\hat{w}_h|_{s_h}\lesssim h^{r+1}.
	\end{equation}
\end{lemma}
\begin{proof}
Applying classical HHO theory (see, e.g., \cite[Thm.~2.27 \& 2.28]{Pietro2020}), yields an error estimate of the form \cref{eq:err1} for the HHO approximation \( \hat{v}_h \). To transfer this estimate to its \( L^2 \)-normalised counterpart \( \hat{w}_h \), we observe that
\[
\| v_{\mathcal{T}_h} - u \|_{L^2} 
\leq \| v_{\mathcal{T}_h} - \Pi_{\mathcal T_h}^{k+1} u \|_{L^2} + \| \Pi_{\mathcal T_h}^{k+1} u - u \|_{L^2}\lesssim h^{r+1},
\]
where we used \cref{prop:discrete:sobolev:embed}  along with the approximation properties of~\( \Pi_{\mathcal T_h}^{k+1} \).
Due to the $L^2$-normalization condition \( \| u \|_{L^2} = 1 \), it follows that \( \left| \| v_{\mathcal{T}_h} \|_{L^2} - 1 \right| \lesssim h^{r+1} \), which readily implies the assertion.
\end{proof}

It remains to bound the second error, which is done in the following lemma.

\begin{lemma}[Bound of second term]\label{lem:err_est}
		Assume that, for some $0\leq r\leq k$, it holds that $u\in H^{r+2}(\mathcal{T}_h)$. Then, we have the following estimate:
	\begin{align*}
		&\|u_{\mathcal{T}_h} - w_{\mathcal{T}_h}\|_{L^2} + \|\nabla_h\mathcal R_h(\hat{u}_h-\hat{w}_h)\|_{L^2} + |\hat{u}_h|_{s_h} + |\hat{u}_h-\hat{w}_h|_{s_h}\\ & \ \ \ \ \ \ \lesssim \|(u_{\mathcal{T}_h} - u)^2\|_{L^3} + h\|u_{\mathcal{T}_h} - u\|_{L^2} + h|\lambda_h-\lambda|+ h^{r+1}.
	\end{align*}
\end{lemma}

\begin{proof}
In this proof, we write pairs of functions in \( L^2(\Omega) \times H^1(\mathcal{T}_h) \) in Roman boldface letters, e.g., \( \mathbf{u} \), \( \mathbf{v} \), and \( \mathbf{w} \). Given the ground state \( u \) and the corresponding eigenvalue \( \lambda \), we define the bilinear form \( J_{u,\lambda} \) for any \( \mathbf{v} = (v, \gamma) \) and \( \mathbf{w} = (w, \tau) \) as
\begin{equation*}
	J_{u,\lambda}(\mathbf{v}, \mathbf{w}) \coloneqq (\nabla_h \gamma, \nabla_h \tau)_{\Omega} + (Vv, w)_{\Omega} + 3\kappa(|u|^2 v, w)_{\Omega} - \lambda (v, w)_{\Omega}.
\end{equation*}
Further, denote \( \mathbf{u}_h \coloneqq  (u_{\mathcal{T}_h}, \mathcal{R}_h \hat{u}_h) \) and \( \mathbf{w}_h \coloneqq (w_{\mathcal{T}_h}, \mathcal{R}_h \hat{w}_h) \), and abbreviate \( \mathbf{y}_h \coloneqq \mathbf{u}_h - \mathbf{w}_h \) and \( \hat{y}_h \coloneqq \hat{u}_h - \hat{w}_h \). The function \( \hat{y}_h \) can be seen as the HHO approximation to the weak solution of the auxiliary Poisson problem: seek \( \chi \in H^1_0(\Omega) \) such that
	\begin{equation*}
	-\Delta \chi = f_h-{f}/\|v_{\mathcal{T}_h}\|_{L^2} \eqqcolon g_h
	\end{equation*}
with \( f \) and \( f_h \) defined in \cref{eq:auxpoisson} and \cref{eq:ustar}, respectively. 
By classical elliptic regularity, \( \chi \in H^2(\Omega) \cap H_0^1(\Omega) \) with \( \|\chi\|_{H^2} \lesssim \|g_h\|_{L^2} \lesssim 1 \), since \( g_h \) is uniformly $L^2$-bounded. Applying classical HHO convergence results for the Poisson problem (see, e.g., \cite[Thm.~2.27 \& 2.28]{Pietro2020}) along with  \cref{prop:discrete:sobolev:embed}, we obtain that
\begin{equation}\label{eq:HHO:auxi_err}
	\|y_{\mathcal{T}_h} -\chi\|_{L^2} + \|\nabla_h(\mathcal R_h\hat{y}_h-\chi)\|_{L^2}\lesssim h\|\chi\|_{H^2}.
\end{equation}

To estimate the discrete errors associated with \( \hat{y}_h \), we use the identity
\[
\|y_{\mathcal{T}_h}\|_{L^2}^2 = \|\chi\|_{L^2}^2 + (y_{\mathcal{T}_h} - \chi, y_{\mathcal{T}_h} + \chi)_{L^2},
\]
and a similar one for the term \( \|\nabla_h \mathcal{R}_h \hat{y}_h\|_{L^2}^2 \). Applying \cref{eq:HHO:auxi_err} and the weighted Young's inequality with parameter \( \epsilon > 0 \), this leads to the estimate
\begin{align}
	\label{eq:esty}
	\begin{split}
			&\|y_{\mathcal{T}_h} \|_{L^2}^2 + \|\nabla_h\mathcal R_h\hat{y}_h\|_{L^2}^2\\
		&\quad\lesssim \|\chi\|_{L^2}^2 + \|\nabla\chi\|_{L^2}^2  + h\|\chi\|_{H^2}(h\|\chi\|_{H^2} + \|y_{\mathcal T_h}\|_{L^2} + \|\nabla_h \mathcal R_h \hat y_h\|_{L^2})\\
		&\quad \leq\|\chi\|_{L^2}^2 + \|\nabla\chi\|_{L^2}^2  +(1+\tfrac{1}{4\epsilon})h^2\|\chi\|_{H^2}^2 + \epsilon(\|y_{\mathcal T_h}\|_{L^2}^2 + \|\nabla_h \mathcal R_h \hat y_h\|_{L^2}^2).
	\end{split}
\end{align}
Choosing the parameter \( \epsilon \) sufficiently small (independently of \( h \)) allows the last term on the right-hand side to be absorbed into the left-hand side, yielding
\begin{align*}
	\|y_{\mathcal{T}_h} \|_{L^2}^2 + \|\nabla_h \mathcal{R}_h \hat{y}_h\|_{L^2}^2 
	\lesssim \|\chi\|_{L^2}^2 + \|\nabla \chi\|_{L^2}^2 + h^2 \|\chi\|_{H^2}^2 
	\lesssim J_{u,\lambda}(\mathbf{x}, \mathbf{x}) + h^2 \|\chi\|_{H^2}^2,
\end{align*}
where the second estimate follows from \cite[Lem.~1]{CCM10} with \( \mathbf{x} \coloneqq (\chi, \chi) \).
The referenced result applies only in the conforming setting, and thus for the pair~\( \mathbf{x} \), but not directly for its discrete counterpart \( \mathbf{y}_h \). To relate back to \( \mathbf{y}_h \), we invoke estimate~\cref{eq:HHO:auxi_err}, the \( L^{\infty} \)-bound for \( u \), and the weighted Young’s inequality with parameter \( \delta > 0 \). Proceeding similarly to  estimate~\cref{eq:esty}, we obtain
\begin{align*}
	|J_{u,\lambda}(\mathbf{y}_h, \mathbf{y}_h) - J_{u,\lambda}(\mathbf{x}, \mathbf{x})|
	\lesssim{} & \left(1 + \tfrac{1}{4\delta}\right) h^2 \|\chi\|_{H^2}^2 + \delta \left( \|y_{\mathcal{T}_h} \|_{L^2}^2 + \|\nabla_h \mathcal{R}_h \hat{y}_h\|_{L^2}^2 \right).
\end{align*}
Combining  this with the previous estimate and choosing \( \delta \) sufficiently small (independent of \( h \)) allows to absorb the last term into the left-hand side, yielding
\begin{align}\label{eq:err:auxi_1}
	\begin{split}
		\|y_{\mathcal{T}_h} \|_{L^2}^2 + \|\nabla_h\mathcal R_h\hat{y}_h\|_{L^2}^2 &\lesssim  J_{u,\lambda}(\mathbf{y}_h,\mathbf{y}_h) + h^2 \|\chi\|_{H^2}^2\\
	&=  \underbrace{J_{u,\lambda}(\mathbf{u}_h-\mathbf{u},\mathbf{y}_h )}_{\eqqcolon \Xi_1}+ \underbrace{J_{u,\lambda}(\mathbf{u}-\mathbf{w}_h,\mathbf{y}_h)}_{\eqqcolon \Xi_2} + \underbrace{h^2 \|\chi\|_{H^2}^2}_{\eqqcolon \Xi_3},
		\end{split}
\end{align}
where we denote $ \mathbf{u} \coloneqq (u,u)$.

To prepare the estimate for \( \Xi_1 \), we observe the following identity:
\begin{align*}
	J_{u,\lambda}(\mathbf{u}_h,\mathbf{y}_h )  
	 =  (\lambda_h -\lambda )(u_{\mathcal{T}_h},y_{\mathcal{T}_h})_{L^2} + \kappa(3u^2u_{\mathcal{T}_h} - u_{\mathcal{T}_h}^3,y_{\mathcal{T}_h})_{L^2}  - s_h(\hat{u}_h,\hat{y}_h),
\end{align*}
where we added and subtracted the terms \( s_h(\hat{u}_h, \hat{y}_h) \) and \( \kappa (u_{\mathcal{T}_h}^2, y_{\mathcal{T}_h})_{L^2} \), and applied~\cref{eq:eigen_pro:discrete} with the test function \( \hat{y}_h \). Furthermore, we have the identity
\begin{align*}
	J_{u,\lambda}(\mathbf{u},\mathbf{y}_h ) 
	 = (\nabla u, \nabla_h (\mathcal R_h\hat{y}_h - \mathcal J_h\hat{y}_h))_{L^2} + (\Delta u,y_{\mathcal{T}_h}- \mathcal J_h\hat{y}_h)_{L^2} + \kappa(2u^3,y_{\mathcal{T}_h})_{L^2},
\end{align*}
where we added and subtracted terms involving the moment-preserving smoother \( \mathcal{J}_h \colon \hat{U}_h \to H^1_0(\Omega) \) from  \cref{prop:prese_inter}, allowing us to apply \cref{eq:eigen_pro:conti} with \( \mathcal{J}_h \hat{y}_h \) as test function, and we used the identity \( Vu + \kappa u^3 - \lambda u = \Delta u \).
Combining the two identities above and noting that
\[
(u_{\mathcal{T}_h}, y_{\mathcal{T}_h})_{L^2} = \tfrac{1}{2} \|u_{\mathcal{T}_h} - w_{\mathcal{T}_h}\|_{L^2}^2, \quad
s_h(\hat{u}_h, \hat{y}_h) = \tfrac{1}{2} \big( |\hat{u}_h|_{s_h}^2 + |\hat{u}_h - \hat{w}_h|_{s_h}^2 - |\hat{w}_h|_{s_h}^2 \big),
\]
where the first relation follows from the fact that \( \|u_{\mathcal{T}_h}\|_{L^2} = \|w_{\mathcal{T}_h}\|_{L^2} \), yields that
\begin{multline*}
	\Xi_1 
	=  -\kappa((u-u_{\mathcal{T}_h})^2(2u+u_{\mathcal{T}_h}),y_{\mathcal{T}_h})_{L^2} - \tfrac{1}{2}(|\hat{u}_h|_{s_h}^2 + |\hat{y}_h|_{s_h}^2 - |\hat{w}_h|_{s_h}^2) \\ +\tfrac{1}{2}(\lambda_h-\lambda)\|y_{\mathcal{T}_h}\|_{L^2}^2-(\nabla u, \nabla_h (\mathcal R_h\hat{y}_h - \mathcal J_h\hat{y}_h))_{L^2} - (\Delta u,y_{\mathcal{T}_h}- \mathcal J_h\hat{y}_h)_{L^2}.
\end{multline*}

Next, we reformulate inequality \cref{eq:err:auxi_1} with the help of the latter identity, where we move the terms $|\hat{u}_h|_{s_h}^2$ and $|\hat{y}_h|_{s_h}^2$ to the left-hand side. This gives that
\begin{align}\label{eq:err:auxi_2}
		&\|y_{\mathcal{T}_h}\|_{L^2}^2 + \|\nabla_h\mathcal R_h\hat{y}_h\|_{L^2}^2 + |\hat{u}_h|_{s_h}^2 + |\hat{y}_h|_{s_h}^2\notag\\
		& \   \ \ \lesssim  |((u-u_{\mathcal{T}_h})^2(2u+u_{\mathcal{T}_h}),y_{\mathcal{T}_h})_{L^2}| + |\hat{w}_h|_{s_h}^2 + |\lambda_h-\lambda|\|y_{\mathcal{T}_h}\|_{L^2}^2\\
		& \ \  \ \ \  \ \  \ +|(\nabla u, \nabla_h (\mathcal R_h\hat{y}_h - \mathcal J_h\hat{y}_h))_{L^2}| +  |(\Delta u,y_{\mathcal{T}_h}- \mathcal J_h\hat{y}_h)_{L^2}| +|\Xi_2| + |\Xi_3|.\notag
\end{align}
	In the following, we estimate the terms on the right-hand side individually. For the first term, the uniform \( L^6 \)-boundedness of both \( u \) and \( u_{\mathcal{T}_h} \) implies that
\begin{equation*}\label{eq:err_est_auxi_3}
	|((u-u_{\mathcal{T}_h})^2(2u+u_{\mathcal{T}_h}),y_{\mathcal{T}_h})_{L^2}| \lesssim \|(u-u_{\mathcal{T}_h})^2\|_{L^3}\|y_{\mathcal{T}_h}\|_{L^2}.
\end{equation*}
The second term on the right-hand side of \cref{eq:err:auxi_2} is estimated using  \cref{lem:firstterm}, while the third term can be absorbed into the left-hand side for \( h \) sufficiently small.

Using the properties of \( \mathcal{J}_h \) from \cref{prop:prese_inter}, we estimate the next two terms as
\begin{multline*}\label{eq:err_est_auxi_1}
		|(\nabla u, \nabla_h (\mathcal R_h\hat{y}_h - \mathcal J_h\hat{y}_h))_{L^2}| 
		= | (\nabla_h (u-\mathcal G_h u), \nabla_h \mathcal J_h\hat{y}_h)_{L^2}|
		\lesssim h^{r+1} \|\hat{y}_h\|_{a_h},\\
			|(\Delta u,y_{\mathcal{T}_h}- \mathcal J_h\hat{y}_h)_{L^2}| 
		=  |(h(\mathtt{id}-\Pi_{\mathcal T_h}^{k+1})\Delta u,h^{-1}(y_{\mathcal{T}_h}- \mathcal J_h\hat{y}_h))_{L^1}|
		\lesssim h^{r+1} \|\hat{y}_h\|_{a_h},
\end{multline*}
where, for the latter estimate, we used that \( u \in H^{r+2}(\mathcal{T}_h) \), which implies that \( \Delta u \in H^r(\mathcal{T}_h) \), in combination with the estimate
\begin{equation*}
	\|h^{-1}(y_{\mathcal{T}_h}- \mathcal J_h\hat{y}_h)\|_{L^2}  = \|h^{-1}(
	\mathtt{id}-\Pi_{\mathcal T_h}^{k+1})\mathcal J_h\hat{y}_h\|_{L^2} \lesssim \|
\hat{y}_h\|_{a_h},
\end{equation*}
by the approximation properties of $\Pi_{\mathcal{T}_h}^{k+1}$.

Therefore, it only remains to estimate $\Xi_2$ and $\Xi_3$. 
From the continuity of the bilinear form $J_{u,\lambda}$ and \cref{lem:firstterm}, we obtain for $\Xi_2$ that
\begin{equation*}\label{eq:err_est_auxi_5}
	|\Xi_2|\lesssim h^{r+1}(\|y_{\mathcal{T}_h} \|_{L^2} + \|\nabla_h\mathcal R_h\hat{y}_h\|_{L^2}).
\end{equation*}
Finally, the term \( \Xi_3 \) can be estimated as
\begin{equation*}\label{eq:err_est_auxi_6}
		|\Xi_3| 
		\lesssim h^2 \| f_h - {f}\big/{\|v_{\mathcal{T}_h}\|_{L^2}} \|_{L^2}^2
		\lesssim h^2 \|u_{\mathcal{T}_h} - u\|_{L^2}^2 + h^2 |\lambda_h - \lambda|^2 + h^{2r+4},
\end{equation*}
where \( f \) and \( f_h \) are defined in \cref{eq:auxpoisson} and \cref{eq:ustar}, respectively. Here, we used that \( |1 - \|v_{\mathcal{T}_h}\|_{L^2}| \lesssim h^{r+1} \), as shown in the proof of  \cref{lem:firstterm}, and that both~\( u \) and~\( u_{\mathcal{T}_h} \) are uniformly \( L^4 \)-bounded.

Combining the above estimates, we obtain
\begin{align*}
	&\|y_{\mathcal{T}_h}\|_{L^2}^2 + \|\nabla_h\mathcal R_h\hat{y}_h\|_{L^2}^2 + |\hat{u}_h|_{s_h}^2 + |\hat{y}_h|_{s_h}^2\\
	& \ \ \ \ \ \ \lesssim 
	 h^{r+1}(\|\hat{y}_h\|_{a_h}  +\|y_{\mathcal{T}_h} \|_{L^2}) + \|(u-u_{\mathcal{T}_h})^2\|_{L^3}\|y_{\mathcal{T}_h}\|_{L^2}\\
	 &\ \ \ \ \ \ \ \ \ \ \ \  + h^2\|u_{\mathcal{T}_h}-u\|_{L^2}^2 + h^2|\lambda_h-\lambda|^2 + h^{2r+2} + h^{2r+4},
\end{align*}
and the assertion can be concluded using the weighted Young's inequality.
\end{proof}

The following theorem gives a convergence result for the ground state, energy, and eigenvalue approximations of the proposed HHO method.
\begin{theorem}[A priori error estimate]\label{theo:err_est}
Assume that \( u \in H^{r+2}(\mathcal{T}_h) \) for some \( 0 \leq r \leq k \). Then, the following approximation results hold for the ground state:
	\begin{equation}\label{eq:energy:err_est}
		\|u_{\mathcal{T}_h} - u\|_{L^2} + \|\nabla_h(\mathcal R_h\hat{u}_h-u)\|_{L^2} + |\hat{u}_h|_{s_h} + |\hat{u}_h-\mathcal I_hu|_{s_h}\lesssim h^{r+1}.
	\end{equation}
The eigenvalue and energy approximations further satisfy
\begin{equation}
	\label{eq:energylamerr}
	|\lambda - \lambda_h| \lesssim h^{r+1}, \quad |E - E_h| \lesssim h^{2r+ 2}.
\end{equation}
\end{theorem}
\begin{proof}
	Using the triangle inequality and  \cref{lem:err_est,lem:firstterm}, we obtain
	\begin{align}
		\label{eq:estapriori}
		\begin{split}
				&\|u_{\mathcal{T}_h} - u\|_{L^2} + \|\nabla_h(\mathcal R_h\hat{u}_h-u)\|_{L^2} + |\hat{u}_h|_{s_h} + |\hat{u}_h-\mathcal I_hu|_{s_h}\\
			& \ \ \ \ \ \ \lesssim h^{r+1} + \|(u-u_{\mathcal{T}_h})^2\|_{L^3} + h\|u_{\mathcal{T}_h}-u\|_{L^2} + h|\lambda_h-\lambda|.
		\end{split}
		\end{align}
In the following, we consider the terms on the right-hand side of the latter inequality individually. For the second term, we obtain with the triangle inequality that
\begin{align}
	\label{eq:trie}
	\|(u-u_{\mathcal{T}_h})^2\|_{L^3} \lesssim &
	  \|u-\Pi_{\mathcal T_h}^{k+1}u\|_{L^6}^2 + \|\Pi_{\mathcal T_h}^{k+1}u - u_{\mathcal{T}_h}\|_{L^6}^2,
\end{align}
where the first term can estimated using classical approximation results for the $L^2$-projection $\Pi_{\mathcal T_h}^{k+1}$ in the $L^6$-norm, as well as the (broken) Sobolev embedding $W^{r+1,6}(\mathcal T_h) \hookrightarrow H^{r+2}(\mathcal T_h)$ to obtain that
\begin{align*}
\|u-\Pi_{\mathcal T_h}^{k+1}u\|_{L^6} \lesssim h^{r+1}|u|_{W^{r+1,6}(\mathcal T_h)} \lesssim h^{r+1}\|u\|_{H^{r+2}(\mathcal T_h)}
\end{align*}
For the second term on the right-hand side of \cref{eq:trie}, we use the discrete Sobolev embedding from \cref{prop:discrete:sobolev:embed}, the triangle inequality, \cref{eq:keyidentity}, and the approximation properties of the Galerkin projection to get that
\begin{align*}
	\|\Pi_{\mathcal T_h}^{k+1}u - u_{\mathcal{T}_h}\|_{L^6}^2&\lesssim   \|\hat{u}_h -\mathcal I_hu\|_{a_h}^2\\
	&  \lesssim   \|\nabla_h(\mathcal R_h\hat{u}_h-u)\|_{L^2}^2 + \|\nabla_h(u-\mathcal G_hu)\|_{L^2}^2 + |\hat{u}_h-\mathcal I_hu|_{s_h}^2\\
	&\lesssim h^{2r+2} + \|\nabla_h(\mathcal R_h\hat{u}_h-u)\|_{L^2}^2 + |\hat{u}_h-\mathcal I_hu|_{s_h}^2.
\end{align*}
Using the latter two estimates we can continue \cref{eq:trie} as
\begin{equation}
	\label{eq:squareterm}
		\|(u-u_{\mathcal{T}_h})^2\|_{L^3} \lesssim h^{2r+2} + \|\nabla_h(\mathcal R_h\hat{u}_h-u)\|_{L^2}^2 + |\hat{u}_h-\mathcal I_hu|_{s_h}^2.
\end{equation}
 Thanks to the plain convergence result of \cref{theo:plain_conver}, $\|(u-u_{\mathcal{T}_h})^2\|_{L^3} $ is a higher-order term, and can thus be absorbed in the left-hand side for sufficiently small $h>0$. 
Note that, similarly, also the third term on the right-hand side of \cref{eq:estapriori} can be absorbed into the left-hand side, for $h > 0$ sufficiently small.

To estimate the fourth term on the right-hand side of \cref{eq:estapriori}, we note that
\begin{align}
	\label{eq:lamerr}
	|\lambda-\lambda_h| 
	\lesssim  \|u_{\mathcal{T}_h} - u\|_{L^2} + \|\nabla_h(\mathcal R_h\hat{u}_h-u)\|_{L^2} + |\hat{u}_h|_{s_h}^2,
\end{align}
where we have used that  $\|\hat{u}_h\|_{a_h}$ and $\|u_{\mathcal{T}_h}\|_{L^6}$ are uniformly bounded. Therefore, this term is also of higher order and can be absorbed into the left-hand term for sufficiently
small $h>0$. 
Convergence result \cref{eq:energy:err_est} follows from combining the above estimates. The first estimate in \cref{eq:energylamerr}, which is the desired eigenvalue approximation result,  follows by combining \cref{eq:lamerr,eq:energy:err_est}.

Finally, for proving the desired energy approximation result, we note that
\begin{align*}
	E-E_h = 
	  \tfrac{1}{2}\|\nabla_h(u-\mathcal R_h\hat{u}_h)\|_{L^2}^2 + \tfrac{1}{2}\|V^{1/2}(u-u_{\mathcal{T}_h})\|_{L^2}^2- \tfrac{1}{2}|\hat{u}_h|_{s_h}^2 + R,
\end{align*}
where the first and second terms on the right-hand side are of order $2r+2$, using~\cref{eq:energy:err_est}. In what follows, we will also show that the remainder $R$, defined as
\begin{align*}
	R \coloneqq (\nabla_h(u-\mathcal R_h\hat{u}_h), \nabla_h\mathcal R_h\hat{u}_h)_{L^2} + ( V(u- u_{\mathcal{T}_h}),u_{\mathcal{T}_h})_{L^2} + \tfrac{1}{4}\kappa(\|u\|_{L^4}^4 - \|u_{\mathcal{T}_h}\|_{L^4}^4) 
\end{align*}
is of order $2r+2$. 
To this end, we first rewrite the remainder $R$, using \cref{eq:eigen_pro:discrete} with the test function $\mathcal I_h u - \hat u_h$,  the identity $( \Pi_{\mathcal T_h}^{k+1}u- u_{\mathcal{T}_h},u_{\mathcal{T}_h})_{L^2} 
   = -\tfrac{1}{2}\|u-u_{\mathcal{T}_h}\|_{L^2}^2$, and some further algebraic manipulations, to obtain that
\begin{align*}
	R 
	 &=  \lambda_h ( \Pi_{\mathcal T_h}^{k+1}u- u_{\mathcal{T}_h},u_{\mathcal{T}_h})_{L^2} - \kappa(u_{\mathcal{T}_h}^3, \Pi_{\mathcal T_h}^{k+1}u- u_{\mathcal{T}_h})_{L^2} + \tfrac{1}{4}\kappa(\|u\|_{L^4}^4 - \|u_{\mathcal{T}_h}\|_{L^4}^4) \\ & \ \ \ \ \ \ \ \ -s_h(\hat{u}_h,\mathcal I_hu-\hat{u}_h) +  (V(u - \Pi_{\mathcal T_h}^{k+1}u),u_{\mathcal{T}_h})_{L^2}\\
	 & = -\tfrac{1}{2}\lambda_h\|u - u_{\mathcal{T}_h}\|_{L^2}^2  +  \tfrac{\kappa}{4}((u_{\mathcal{T}_h}-u)^2,3u_{\mathcal{T}_h}^2 + 2u_{\mathcal{T}_h}u +u^2)_{L^2} \\ 
	 & \ \ \ \ \ \ \ \ -s_h(\hat{u}_h,\mathcal I_hu-\hat{u}_h) +  (u - \Pi_{\mathcal T_h}^{k+1}u,Vu_{\mathcal{T}_h} + \kappa u_{\mathcal{T}_h}^3)_{L^2}.
\end{align*}
Using \cref{lem:err_est} and \cref{eq:energy:err_est}, it directly follows that all terms except the last one are of order \(2r+2\). To show that also the last term is of this order, we observe that
\begin{align}
	&(\kappa u_{\mathcal{T}_h}^3 +Vu_{\mathcal{T}_h}, u-\Pi_{\mathcal T_h}^{k+1}u)_{L^2}\notag\\ 
	 &\quad= (\Delta u + \lambda u,  u-\Pi_{\mathcal T_h}^{k+1}u)_{L^2} + (\kappa u_{\mathcal{T}_h}^3 +Vu_{\mathcal{T}_h}-\kappa u^3 -Vu, u-\Pi_{\mathcal T_h}^{k+1}u)_{L^2}\label{eq:2rp2},
\end{align}
where we have used \cref{eq:auxpoisson}. 
Since, by assumption, \(u \in H^{r+2}(\mathcal{T}_h)\), it follows that \(\Delta u \in H^r(\mathcal{T}_h)\), allowing us to apply the orthogonality and approximation properties of \(\Pi^{k+1}_{\mathcal{T}_h}\) to show that the first term is of order \(2r+2\).
A similar bound for the second term is obtained directly from \cref{eq:energy:err_est}.
 Consequently, all terms in \cref{eq:2rp2} are of order \(2r+2\), and the energy approximation estimate in \cref{eq:energylamerr} follows.
\end{proof}

The $L^2$- and eigenvalue approximation estimates stated in \cref{eq:energylamerr,eq:energy:err_est} can be improved as outlined in the following theorem.

\begin{theorem}[Improved error estimate]\label{theo:l2errorest}
	
	Assume that \( u \in H^{r+2}(\mathcal{T}_h) \) for some \( 0 \leq r \leq k \). Then, the following $L^2$-approximation results hold for the ground state:	
	\begin{equation}
		\label{eq:improvedL2errest}
		\|u_{\mathcal T_h}-u\|_{L^2} + \|\mathcal R_h\hat{u}_h-u\|_{L^2} \lesssim h^{r+2}.
	\end{equation}
	Furthermore, we can improve the eigenvalue approximation estimate to
	\begin{equation}
		\label{eq:improvedev}
		|\lambda-\lambda_h|\lesssim h^{r+2}.
	\end{equation}
\end{theorem}
\begin{proof}
We begin by introducing an auxiliary problem. For any given \( w \in L^2(\Omega) \), it seeks \( \psi_w \in H_0^1(\Omega) \) such that the following equation holds in the weak sense:
\begin{equation*}
	-\Delta \psi_w + (V + 3\kappa u^2 - \lambda)\psi_w = 2\kappa (u^3, \psi_w)_{\Omega} \, u + w - (w, u)_{\Omega} \, u.
\end{equation*}
It can be verified that this problem is solved by the unique solution $\psi_{w} \in  u^{\perp}:=\{v\in H_0^1(\Omega): (u,v)_{\Omega}=0\}$ satisfying, for all  $v \in u^{\perp}$, the variational problem
	\begin{equation}\label{eq:auxi_elliptic}
			(\nabla \psi_{w},\nabla v)_{L^2} + ((V+ 3\kappa u^2-\lambda)\psi_{w},v)_{L^2} 
			=  (w,v)_{L^2}.
	\end{equation}
	The well-posedness of the latter problem follows from the Lax--Milgram theorem using the coercivity of the bilinear form on the left-hand side (cf.~\cite[Lem.~1]{CCM10})  and the fact that \( u^\perp \) is a closed subspace of \( H_0^1(\Omega) \). Classical elliptic regularity then gives that \( \psi_w \in H^2(\Omega) \cap H_0^1(\Omega) \) with the estimate  $\|\psi_w\|_{H^2} \lesssim \|w\|_{L^2}.$

Next, we define the error $\hat e_h = (e_{\mathcal{T}h}, e_{\mathcal{F}_h}) := \hat{u}_h - \mathcal{I}_h u$, test equation \cref{eq:auxi_elliptic} with the function $\mathcal{J}_h \hat{e}_h - (\mathcal{J}_h \hat{e}_h, u)_{L^2} u \in u^\perp$, where $\mathcal J_h$ denotes the moment preserving operator from \cref{prop:prese_inter}, and set $w \coloneqq e_{\mathcal{T}_h}$. This yields that
\begin{align*}
	\|e_{\mathcal{T}_h}\|_{L^2}^2 
	&= (e_{\mathcal{T}_h}, \mathcal{J}_h \hat{e}_h - (\mathcal{J}_h \hat{e}_h, u)_{L^2} \, u )_{L^2} 
	+ (\mathcal{J}_h \hat{e}_h, u)_{L^2} \, (e_{\mathcal{T}_h}, u)_{L^2} \\
	&= \underbrace{(\nabla \psi_{e_{\mathcal{T}_h}}, \nabla \mathcal{J}_h \hat{e}_h)_{L^2} 
	+ ( (V + 3\kappa u^2 - \lambda)\, \psi_{e_{\mathcal{T}_h}}, \mathcal{J}_h \hat{e}_h )_{L^2}}_{\eqqcolon \Xi_1} \\
	&\qquad - \underbrace{(\mathcal{J}_h \hat{e}_h, u)_{L^2} 
	\big( (\nabla \psi_{e_{\mathcal{T}_h}}, \nabla u)_{L^2} 
	+ ((V + 3\kappa u^2 - \lambda)\, \psi_{e_{\mathcal{T}_h}}, u)_{L^2} \big)}_{\eqqcolon \Xi_2} \\
	&\qquad  + \underbrace{(\mathcal{J}_h \hat{e}_h, u)_{L^2} \, (e_{\mathcal{T}_h}, u)_{L^2}}_{\eqqcolon \Xi_3}.
\end{align*}

Before estimating terms $\Xi_1$--$\Xi_3$, we first derive a bound for $|(\mathcal{J}_h \hat{e}_h, u)_{L^2}|$. To this end, we use the properties of $\mathcal{J}_h$ as stated in \cref{prop:prese_inter}, which yields that
\begin{align}
	\label{eq:L2innterproductJeu}
	(\mathcal J_h\hat{e}_h,u)_{L^2} 
	 = (e_{\mathcal{T}_h},u_{\mathcal{T}_h})_{L^2} + (\mathcal J_h\hat{e}_h,u-u_{\mathcal{T}_{h}})_{L^2}.
\end{align}
The two terms on the right-hand side  can be individually estimated as
\begin{align*}
	|(e_{\mathcal{T}_h},u_{\mathcal{T}_h})_{L^2}| &=
	  |(u_{\mathcal{T}_h}-u,u_{\mathcal{T}_h})_{L^2}| = \tfrac{1}{2}\|u-u_{\mathcal{T}_h}\|_{L^2}^2\lesssim h^{2r+2},\\
		|(\mathcal J_h\hat{e}_h,u-u_{\mathcal{T}_{h}})_{L^2}|&\leq \|\mathcal J_h\hat{e}_h\|_{L^2}\|u-u_{\mathcal{T}_h}\|_{L^2}\lesssim \|\hat{e}_h\|_{a_h} \|u-u_{\mathcal{T}_h}\|_{L^2}\lesssim h^{2r+2},
\end{align*}
where we used the $H^1$-continuity of $\mathcal J_h$ and the error estimates provided in \cref{theo:err_est}. 
By inserting the latter two bounds into \cref{eq:L2innterproductJeu}, we obtain that
\begin{equation*}
	|(\mathcal J_h\hat{e}_h,u)_{L^2}|\lesssim h^{2r+2}.
\end{equation*}
Combining this estimate with the $L^\infty$-bound for $u$ and recalling the $H^2$-regularity estimate $\|\psi_{e_{\mathcal{T}_h}}\|_{H^2} \lesssim \|e_{\mathcal{T}_h}\|_{L^2}$, we obtain the following bounds for $\Xi_2$ and $\Xi_3$:
\begin{align*}
	|\Xi_2| \lesssim h^{2r+2}\|\psi_{e_{\mathcal{T}_h}}\|_{H^1} \lesssim h^{2r+2}\|e_{\mathcal{T}_h}\|_{L^2},\qquad 
	|\Xi_3| \lesssim h^{2r+2}\|e_{\mathcal{T}_h}\|_{L^2}.
\end{align*}

For estimating term $\Xi_1$, we introduce the function
\[
\hat{\phi}_h = (\phi_{\mathcal{T}_h}, \phi_{\mathcal{F}_h}) := \mathcal{I}_h \left( \psi_{e_{\mathcal{T}_h}} - \frac{(\psi_{e_{\mathcal{T}_h}}, \Pi_{\mathcal{T}_h}^{k+1} u)_{L^2}}{(u, \Pi_{\mathcal{T}_h}^{k+1} u)_{L^2}} u \right) \in \hat{U}_h,
\]
for which, by construction, it holds that $(\phi_{\mathcal{T}_h}, u)_{L^2} = 0$. To estimate terms involving~$\hat{\phi}_h$, the following bounds will play an important role:
\begin{align}
	\label{eq:boundsforphih1}
			|(\psi_{e_{\mathcal{T}_h}},\Pi_{\mathcal T_h}^{k+1}u)_{L^2}| &
		= |(\psi_{e_{\mathcal{T}_h}} -\Pi_{\mathcal T_h}^1\psi_{e_{\mathcal{T}_h}} ,\Pi_{\mathcal T_h}^{k+1}u-u)_{L^2}|\lesssim h^{r+4}\|\psi_{e_{\mathcal{T}_h}}\|_{H^2},	\\
		|(u,\Pi_{\mathcal T_h}^{k+1}u)_{L^2}| &=|(\Pi_{\mathcal T_h}^{k+1}u,\Pi_{\mathcal T_h}^{k+1}u)_{L^2}| = |1 - \|u-\Pi_{\mathcal T_h}^{k+1}u\|_{L^2}^2|\gtrsim 1- \tfrac{h^4}{\pi^4}.\label{eq:boundsforphih2}
\end{align}
The proof of these bounds relies on the fact that $\psi_{e_{\mathcal{T}_h}} \in u^\perp$, together with the approximation and orthogonality properties of $\Pi_{\mathcal{T}_h}^{k+1}$ and the Pythagorean identity. Direct calculations based on \cref{eq:boundsforphih1,eq:boundsforphih2}, which are omitted here for the sake of brevity, then yield the following approximation results:
\begin{align}
	\label{eq:estphih1}
	\| \nabla_h(\psi_{e_{\mathcal{T}_h}} - \mathcal R_h \hat \phi_h)\|_{L^2} +
	 \|\nabla(\mathcal J_h\hat\phi_h - \psi_{e_{\mathcal{T}_h}}) \|_{L^2}
+	|\phi_h|_{s_h}&\lesssim
	 h \|\psi_{e_{\mathcal{T}_h}}\|_{H^2},\\
	 	\|\psi_{e_{\mathcal{T}_h}} -  \phi_{\mathcal{T}_h}\|_{L^2} 
	 +
	\|\mathcal J_h\hat\phi_h - \psi_{e_{\mathcal{T}_h}} \|_{L^2}&\lesssim
	 h^2 \|\psi_{e_{\mathcal{T}_h}}\|_{H^2}.\label{eq:estphih2}
\end{align}

To estimate the term $\Xi_1$, we start by rewriting it as
\begin{equation}\label{eq:Xi1_expansion}
	\begin{aligned}
		\Xi_1 &= (\nabla \mathcal J_h\hat{e}_h, \nabla_h \mathcal R_h \hat{\phi}_h)_{L^2} + ((V + 3\kappa u^2 - \lambda)\mathcal J_h \hat{e}_h, \phi_{\mathcal{T}_h})_{L^2} \\
		&\quad + (\nabla \mathcal J_h \hat{e}_h, \nabla \psi_{e_{\mathcal{T}_h}} - \nabla_h \mathcal R_h\hat \phi_h)_{L^2} + ((V + 3\kappa u^2 - \lambda) \mathcal J_h \hat{e}_h, \psi_{e_{\mathcal{T}_h}} - \phi_{\mathcal{T}_h})_{L^2}.
	\end{aligned}
\end{equation}
Using the properties of $\mathcal{J}_h$ from \cref{prop:prese_inter}, together with \cref{eq:eigen_pro:conti,eq:eigen_pro:discrete} and some algebraic manipulations, we arrive at the identity
\begin{align*}
	&{(\nabla \mathcal J_h\hat{e}_h,\nabla_h\mathcal R_h \hat\phi_h)_{L^2} + ((V+ 3\kappa u^2-\lambda)\mathcal J_h\hat{e}_h , \phi_{\mathcal{T}_h})_{L^2}}\\
	 &\qquad= 
	 \underbrace{((3\kappa u^2 -\kappa u_{\mathcal{T}_h}^2 + \lambda_h-\lambda)u_{\mathcal{T}_h},\phi_{\mathcal{T}_h})_{L^2} - (2\kappa u^3,\phi_{\mathcal{T}_h})_{L^2}}_{\eqqcolon \xi_1}- \underbrace{s_h(\phi_h,\hat{u}_h)}_{\eqqcolon \xi_2}\\
	& \qquad \qquad  + \underbrace{(\nabla_h (u-\mathcal G_hu),\nabla \mathcal J_h \hat\phi_h)_{L^2}}_{\eqqcolon \xi_3} + \underbrace{((V+ \kappa u^2-\lambda)u, \mathcal J_h\hat\phi_{h}- \phi_{\mathcal{T}_h})_{L^2}}_{\eqqcolon \xi_4}\\
	& \qquad \qquad + \underbrace{((V+ 3\kappa u^2-\lambda)(\mathcal J_h\hat{e}_h-e_{\mathcal{T}_h}), \phi_{\mathcal{T}_h})_{L^2}}_{\eqqcolon \xi_5} \\
	&\qquad \qquad+ \underbrace{((V+ 3\kappa u^2-\lambda)(u- \Pi_{\mathcal T_h}^{k+1}u), \phi_{\mathcal{T}_h})_{L^2}}_{\eqqcolon \xi_6},
\end{align*}
where terms $\xi_1$--$\xi_6$ are estimated individually in the following. 

Using \cref{theo:err_est}, \cref{eq:estphih1,eq:estphih2}, the approximation properties of the Galerkin- and $L^2$-projections, and the properties of $\mathcal{J}_h$ from \cref{prop:prese_inter}, we obtain that
\begin{align*}
	|\xi_1| & \leq |\kappa((u-u_{\mathcal{T}_h})^2(2u+u_{\mathcal{T}_h}),\phi_{\mathcal{T}_h})_{L^2}| \\
	&\qquad\quad + |\lambda_h-\lambda| |(u-u_{\mathcal{T}_h},\phi_{\mathcal{T}_h})_{L^2}|
	\lesssim h^{2r+2}\|\psi_{e_{\mathcal{T}_h}}\|_{H^2},\\
	|\xi_2| &\leq |\phi_h|_{s_h}|\hat{u}_h|_{s_h}\lesssim h^{r+2}\|\psi_{e_{\mathcal{T}_h}}\|_{H^2},\\
	|\xi_3| &\leq
	 |(\nabla_h (u-\mathcal G_hu),\nabla_h( \psi_{e_{\mathcal{T}_h}} - \mathcal G_h \psi_{e_{\mathcal{T}_h}} ))_{L^2}|\\
	 &\qquad \quad + |(\nabla_h (u-\mathcal G_hu),\nabla (\psi_{e_{\mathcal{T}_h}}- \mathcal J_h\hat\phi_h))_{L^2}|\lesssim h^{r+2}\|\psi_{e_{\mathcal{T}_h}}\|_{H^2},\\
|\xi_4| 
	&\leq |(\Delta u - \Pi_{\mathcal T_h}^{k+1}\Delta u, \mathcal J_h\hat\phi_{h}- \psi_{e_{\mathcal{T}_h}})_{L^2}| \\
	&\qquad \quad + |(\Delta u - \Pi_{\mathcal T_h}^{k+1}\Delta u, \psi_{e_{\mathcal{T}_h}}- \phi_{\mathcal{T}_h})_{L^2}|\lesssim h^{r+2}\|\psi_{e_{\mathcal{T}_h}}\|_{H^2},\\
	|\xi_5| &\lesssim\|\mathcal J_h\hat{e}_h-e_{\mathcal{T}_h}\|_{L^2}\|\phi_{\mathcal{T}_h}\|_{L^2} \lesssim h\|\nabla \mathcal J_h\hat{e}_h\|_{L^2}\|\phi_{\mathcal{T}_h}\|_{L^2}\lesssim h^{r+2}\|\psi_{e_{\mathcal{T}_h}}\|_{H^2},\\
	|\xi_6| &\lesssim \|u- \Pi_{\mathcal T_h}^{k+1}u\|_{L^2}\|\phi_{\mathcal{T}_h}\|_{L^2}\lesssim h^{r+2}\|\psi_{e_{\mathcal{T}_h}}\|_{H^2}.
\end{align*}
Note that the estimate for $\xi_1$ uses arguments similar to those in \cref{eq:squareterm} together with $(\phi_{\mathcal{T}_h},u)_{L^2}=0$, while the estimate for $\xi_4$ relies on the identity $(V + \kappa u^2 - \lambda) u = \Delta u$. 
Combining these estimates with the $H^2$-regularity bound  \(\|\psi_{e_{\mathcal{T}_h}}\|_{H^2} \lesssim \|e_{\mathcal{T}_h}\|_{L^2},
\) we obtain for the first term on the right-hand side of \cref{eq:Xi1_expansion} that
\begin{equation*}
	|{(\nabla \mathcal J_h\hat{e}_h,\nabla_h\mathcal R_h \hat\phi_h)_{L^2} + ((V+ 3\kappa u^2-\lambda)\mathcal J_h\hat{e}_h , \phi_{\mathcal{T}_h})_{L^2}}|
	\lesssim h^{r+2}\|e_{\mathcal{T}_h}\|_{L^2}.
\end{equation*}

For the second term on the right-hand side of \cref{eq:Xi1_expansion}, we obtain
\begin{align*}
	&|(\nabla \mathcal J_h\hat{e}_h,\nabla\psi_{e_{\mathcal{T}_h}} - \nabla_h\mathcal R_h \hat\phi_h)_{L^2} + ((V+ 3\kappa u^2-\lambda)\mathcal J_h\hat{e}_h , \psi_{e_{\mathcal{T}_h}} -  \phi_{\mathcal{T}_h})_{L^2}|\\
	&\;\lesssim \|\nabla \mathcal J_h\hat{e}_h\|_{L^2}\|\nabla\psi_{e_{\mathcal{T}_h}} - \nabla_h\mathcal R_h \hat\phi_h\|_{L^2} + \|\mathcal J_h\hat{e}_h \|_{L^2}\|\psi_{e_{\mathcal{T}_h}} -  \phi_{\mathcal{T}_h}\|_{L^2}
	\lesssim h^{r+2}\|e_{\mathcal{T}_h}\|_{L^2},
\end{align*}
where we applied \cref{prop:prese_inter} and \cref{eq:estphih1} to estimate the first term, and \cref{eq:estphih2} together with the $H^1$-continuity of $\mathcal{J}_h$ for the second term.

We have now estimated both terms in the expression for $\Xi_1$ from \cref{eq:Xi1_expansion}. Combining these estimates yields the bound
\begin{equation*}
	|\Xi_1| \lesssim h^{r+2} \|e_{\mathcal{T}_h}\|_{L^2},
\end{equation*}
	and the desired improved $L^2$-estimate, which is the first inequality in \cref{eq:improvedL2errest}, follows directly by combining the estimates for $\Xi_1$--$\Xi_3$ and $\|u-\Pi_{\mathcal{T}_h}^{k+1}u\|_{L^2}\lesssim h^{r+2}$. Given the improved $L^2$-estimate,  the refined eigenvalue approximation from \cref{eq:improvedev} follows as
\begin{equation*}
	|\lambda-\lambda_h|\leq 2|E-E_h| + \tfrac{\kappa}{2}|\|u\|_{L^4}^4-\|u_{\mathcal{T}_h}\|_{L^4}^4|\lesssim h^{r+2}.
\end{equation*}

To establish the refined $L^2$-estimate for the reconstructed approximation, we begin with the triangle inequality, which gives 
\begin{align}
	\label{eq:improvedL2rec}
	\|\mathcal R_h\hat{u}_h - u\|_{L^2} \leq \|\mathcal R_h\hat{u}_h - \mathcal R_h\mathcal{I}_h u\|_{L^2} + \|\mathcal{G}_h u - u\|_{L^2}.
\end{align}
For estimating the first term on the right-hand side, we add and subtract the $L^2$-projection onto piecewise constants and apply the Poincaré inequality from \cref{lem:poincare}, along with \cref{theo:err_est} and the approximation properties of the Galerkin projection. This yields the estimate
\begin{align*}
	\|\mathcal R_h\hat{u}_h - \mathcal R_h\mathcal{I}_h u\|_{L^2} 
	&\lesssim \|h \nabla_h(\mathcal R_h\hat{u}_h - \mathcal R_h\mathcal{I}_h u)\|_{L^2} + \|\Pi_{\mathcal{T}_h}^0(\mathcal R_h\hat{u}_h - \mathcal R_h\mathcal{I}_h u)\|_{L^2} \\
	&= \|h \nabla_h(\mathcal R_h\hat{u}_h - \mathcal{G}_h u)\|_{L^2} + \|\Pi_{\mathcal{T}_h}^0(u_{\mathcal{T}_h} - \Pi_{\mathcal{T}_h}^{k+1} u)\|_{L^2} \lesssim h^{r+2}.
\end{align*}
The second term on the right-hand side of \cref{eq:improvedL2rec} can be estimated in the same way, noting that the Galerkin projection preserves element averages by definition, and again using the Poincaré inequality from \cref{lem:poincare}. The refined $L^2$-estimate for the reconstruction, which is the second bound in \cref{eq:improvedL2errest}, follows directly.
\end{proof}

The eigenvalue estimate in \cref{eq:improvedev} appears suboptimal compared to the energy estimate in \cref{eq:energylamerr}, noting that eigenvalue and energy approximations typically converge at the same rate. The following remark shows that, under suitable regularity assumptions, optimal convergence for the eigenvalue can be recovered.

\begin{remark}[Optimal eigenvalue approximation]
Assume that \( u^3 \in H^m(\Omega) \) for some \( 0 \leq m \leq r \), and that the solution \( \psi_{u^3} \) to the dual problem \cref{eq:auxi_elliptic} with right-hand side \( w = u^3 \) satisfies \( \psi_{u^3} \in H^{m+2}(\Omega) \) with the estimate \( \|\psi_{u^3}\|_{H^{m+2}} \lesssim \|u^3\|_{H^m} \). Then there holds the improved eigenvalue approximation result
\[
|\lambda - \lambda_h| \lesssim h^{r + m + 2}.
\]
The proof of this result proceeds similarly to that of the improved $L^2$-error estimate in \cref{theo:l2errorest}, now using dual problem \cref{eq:auxi_elliptic} with the right-hand side \(w = u^3\). This explains the regularity assumptions above. For brevity, the details are omitted here; we refer, for example, to \cite{henning2024discretegroundstatesrotating}, where the corresponding proof is carried out for a classical high-order conforming finite element method.
\end{remark}

Thus far, we have conducted an error analysis of HHO approximation \cref{eq:hhoapprox} to the Gross--Pitaevskii ground state. We now turn our attention to the analysis of its modified variant, defined in \cref{eq:hhoapprox0}, which achieves guaranteed lower bounds for the ground state energy. Recall that only the lowest-order case \( k = 0 \) is of interest for the modified approximation, as the low-order quadrature in~\cref{eq:hhoapprox0} prevents improved convergence rates for \( k > 0 \). Similar as for the classical HHO ground state approximation~\cref{eq:hhoapprox}, also any discrete ground state of its modified variant~\cref{eq:hhoapprox0} satisfies a discrete eigenvalue problem. It seeks an eigenpair $(\hat u_h^0,\lambda_h^0) \in \hat U_h \times \mathbb R$ with $\|u_{\mathcal T_h}^0\|_{L^2} = 1$ such that, for all $\hat v_h \in \hat U_h$, it holds that
\begin{equation}\label{eq:eigen_pro:discrete0}
	\begin{aligned}
		a_h(\hat{u}_h^0, \hat{v}_h) + (V u_{\mathcal{T}_h}^0, v_{\mathcal{T}_h})_{L^2} + \kappa n_h(\hat{u}_h^0, \hat{v}_h) = \lambda_h^0 (u_{\mathcal{T}_h}^0, v_{\mathcal{T}_h})_{L^2},
	\end{aligned}
\end{equation}
where the nonlinearity is encoded in the form \( n_h \), defined for any \(\hat{v}_h, \hat{\varphi}_h \in \hat{U}_h\), as
\begin{equation*}
	n_h(\hat{v}_h, \hat{\varphi}_h) \coloneqq \tfrac{1}{2} \big((\Pi_{\mathcal{T}_h}^0 v_{\mathcal{T}_h})^2 v_{\mathcal{T}_h}, \varphi_{\mathcal{T}_h}\big)_{L^2} + \tfrac{1}{2} \big((v_{\mathcal{T}_h})^2 \Pi_{\mathcal{T}_h}^0 v_{\mathcal{T}_h}, \Pi_{\mathcal{T}_h}^0 \varphi_{\mathcal{T}_h}\big)_{L^2}.
\end{equation*}

The following theorem summarises the convergence results for the modified HHO approximation. As the proof follows that of the standard HHO method, we keep the presentation short and directly state the final result.
\begin{theorem}[A priori error estimate]\label{theo:err_est0}
Assume that \( V \in H^1(\mathcal{T}_h) \). Then, the approximation of the modified HHO method satisfies the following estimates:
\[
\|\nabla_h(\mathcal{R}_h \hat{u}_h^0 - u)\|_{L^2} + |\hat{u}_h^0|_{s_h} + |\hat{u}_h^0 - \mathcal{I}_h u|_{s_h} \lesssim h.
\]
Moreover, the ground state approximation satisfies the  \( L^2 \)-error bounds
\[
\|u_{\mathcal{T}_h}^0 - u\|_{L^2} + \|\mathcal{R}_h \hat{u}_h^0 - u\|_{L^2} \lesssim h^2,
\]
and the energy and eigenvalue approximations satisfy
\[
|E - E_h^0| \lesssim h^2, \quad |\lambda - \lambda_h^0| \lesssim h^2.
\]
\end{theorem}
\begin{proof}As before, one first needs to establish the plain convergence of the modified HHO method. The proof is very similar to that of \cref{theo:plain_conver} for the classical~HHO method and is therefore omitted. As a direct consequence, we obtain the uniform boundedness of the modified discrete energies \( E_h^0 \), as well as the uniform boundedness of \( \|u^0_{\mathcal{T}_h}\|_{L^6} \) via the discrete Sobolev embedding from  \cref{prop:discrete:sobolev:embed}. Furthermore, \cref{lem:linf_bound} ensures the uniform boundedness of \( \|u_{\mathcal{T}_h}^0\|_{L^\infty} \) and \(\|\mathcal{J}_h\hat{u}_h^0\|_{L^{\infty}}\).
Next, repeating the arguments from \cref{lem:firstterm,lem:err_est}, \cref{theo:err_est,theo:l2errorest}, and using \cref{lem:modi_1} to bound the quadrature errors together with \cref{eq:JhIh_linf}, we also obtain the desired convergence properties of the modified HHO approximation.
\end{proof}

\section{Numerical experiments}\label{sec:numexp}

This section presents numerical experiments that validate the theory and demonstrate the method's practical effectiveness.  
To solve the finite-dimensional constrained minimisation problems~\cref{eq:hhoapprox0,eq:hhoapprox}, we employ solvers specifically tailored to their structure.
Of the available approaches (see \cite{HenJ25} for an overview), we employ an adaptation of the (fully discrete) energy-adaptive Sobolev gradient flow from \cite{HenP20}, together with an adaptive choice of step sizes (see Remark 4.3 of that paper). The initial iterate is constructed by setting all interior degrees of freedom to one and all boundary degrees of freedom to zero. The resulting function is normalised with respect to the discrete $L^2$-norm to obtain a suitable initial guess. The iteration is terminated if the relative $L^2$-residual of the current iterate and the relative energy difference between two consecutive iterates falls below $10^{-12}$. The maximal number of iterations is $10^3$. Details on the implementation can be found in the code available at \url{https://github.com/moimmahauck/GPE_HHO_code}, which is based on an implementation of the HHO method used for~\cite{Tran2024}. The latter, in turn, is based on the basic finite element implementation detailed in \cite{Alberty1999}.

All of our numerical experiments consider the domain $\Omega = (-8,8)^2$ and a hierarchy of Friedrichs–Keller triangulations constructed by uniform red refinement of an initial triangulation of the domain consisting of two elements. 
We exclude the coarsest meshes from this hierarchy since they do not resolve the considered potentials and therefore do not yield meaningful approximations. 
To simplify the notation, we denote the side length of the squares formed by joining opposing triangles by~$h$. Since analytical solutions are typically not available, errors are computed with respect to the $Q^{k+1}$-finite element approximation on the uniform Cartesian grid obtained by twice uniform red refinement of the finest mesh in the hierarchy and joining opposing triangles. With a slight abuse of notation, we  denote its energy, eigenvalue, and ground state approximations by $E$, $\lambda$, and $u$, respectively.

\subsection{Guaranteed lower energy bounds}
\begin{figure}
	\begin{minipage}[c]{.32\linewidth}
		\includegraphics[width=\linewidth, height=.85\linewidth]{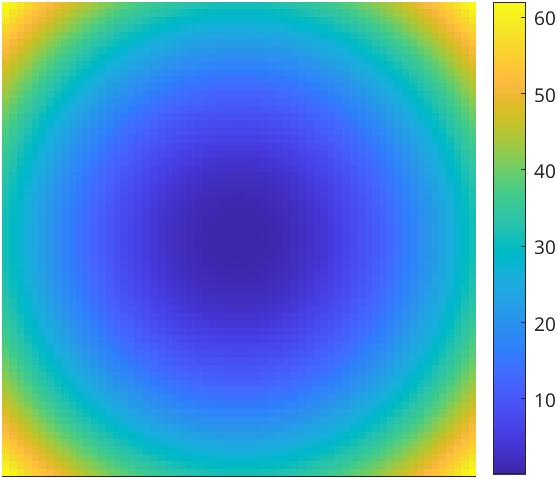}
	\end{minipage}\hfill
	\begin{minipage}[c]{.32\linewidth}
		\includegraphics[width=\linewidth, height=.85\linewidth]{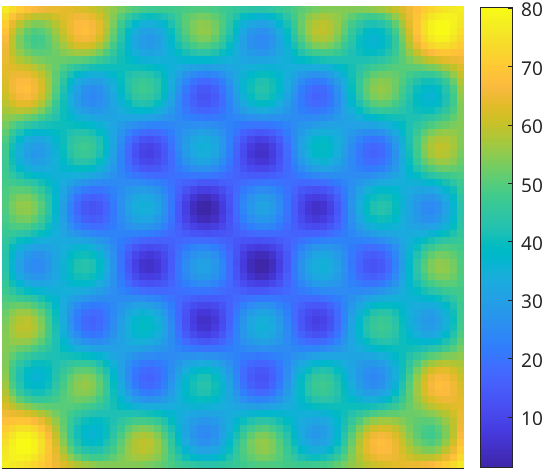}
	\end{minipage}\hfill
	\begin{minipage}[c]{.32\linewidth}
		\includegraphics[width=\linewidth, height=.89\linewidth]{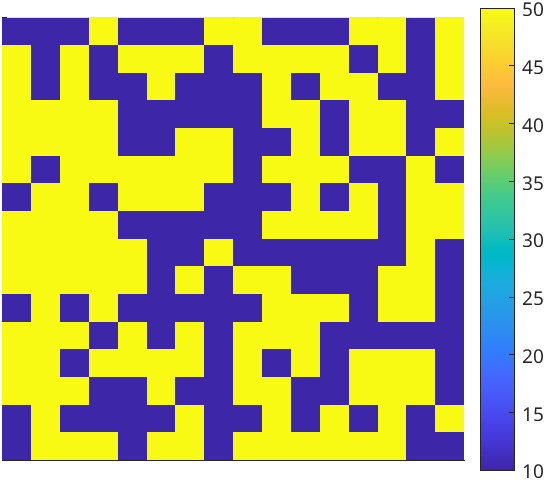}
	\end{minipage}\\[.25cm]
	\begin{minipage}[c]{.32\linewidth}
		\includegraphics[width=\linewidth, height=.85\linewidth]{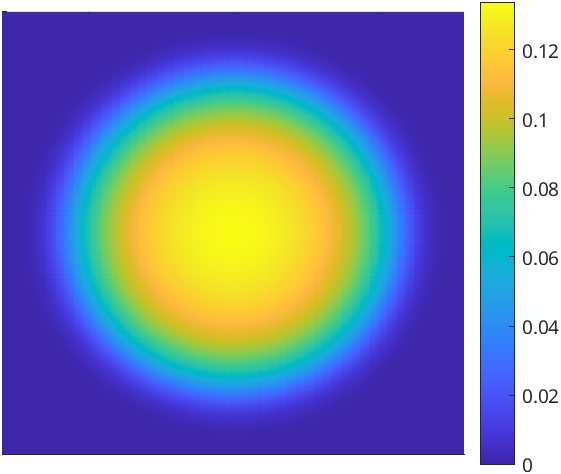}
	\end{minipage}\hfill
	\begin{minipage}[c]{.32\linewidth}
		\includegraphics[width=\linewidth, height=.85\linewidth]{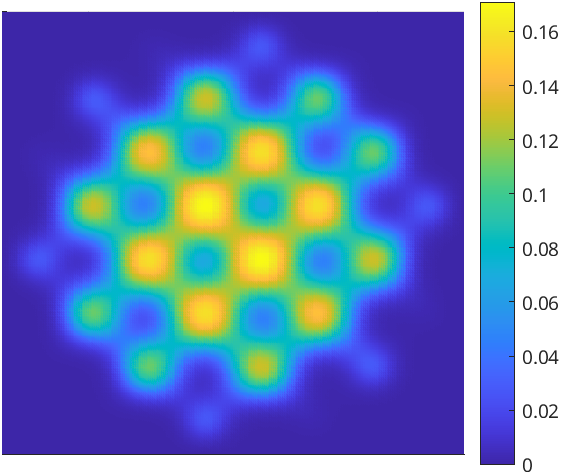}
	\end{minipage}\hfill
	\begin{minipage}[c]{.32\linewidth}
		\includegraphics[width=\linewidth, height=.83\linewidth]{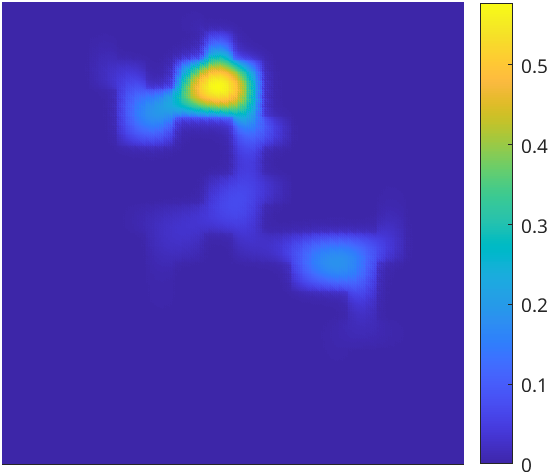}
	\end{minipage}
	\caption{First row: Potentials $V_1$, $V_2$, and~$V_3$ (from left to right). Second row: corresponding ground state approximations.}
	\label{fig:potandgs}
\end{figure}

To investigate the guaranteed lower energy bounds provided by the modified HHO method from \cref{eq:hhoapprox0}, we present three numerical experiments.
The \emph{first experiment} considers the harmonic potential
$
V_1(x) \coloneqq \tfrac{1}{2} |x|^2,
$
with the particle interaction parameter set to \( \kappa = 1000 \).
In the \emph{second experiment}, we use the so-called \emph{lattice potential}, defined as
\[
V_2(x) \coloneqq \tfrac{1}{2} |x|^2 + 15\left(1 + \sin\left(\tfrac{\pi x_1}{2}\right)\sin\left(\tfrac{\pi x_2}{2}\right)\right),
\]
again with \( \kappa = 1000 \).
The \emph{third experiment} involves a disorder potential, denoted by \( V_3 \), which is piecewise constant on a Cartesian grid with mesh size \( h = 2^{0} \). The values on each grid element are assigned randomly as independent realisations of coin-toss variables taking values in \( \{10, 50\} \). The particle interaction parameter is set to \( \kappa = 1 \). Illustrations of the potentials and corresponding ground state approximations are provided in \cref{fig:potandgs}.
Note that for the disorder potential, a phenomenon known as Anderson localisation occurs (see, e.g., \cite{AltHP20,AltHP22}), which leads to an exponential localisation of the ground state.

\begin{figure}
	\begin{minipage}[b]{.32\linewidth}
		\includegraphics[width=\linewidth]{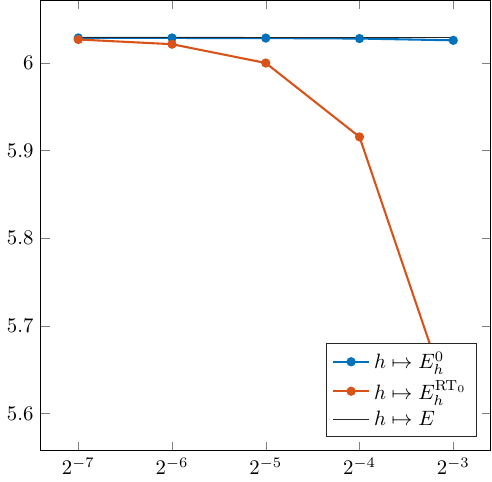}
	\end{minipage}\hfill
	\begin{minipage}[b]{.32\linewidth}
		\includegraphics[width=\linewidth]{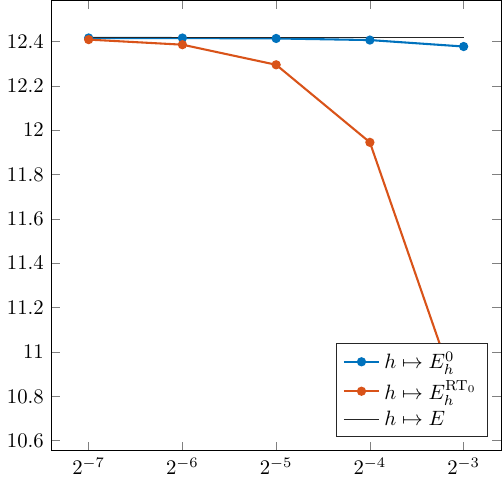}
	\end{minipage}\hfill
	\begin{minipage}[b]{.32\linewidth}
		\includegraphics[width=\linewidth]{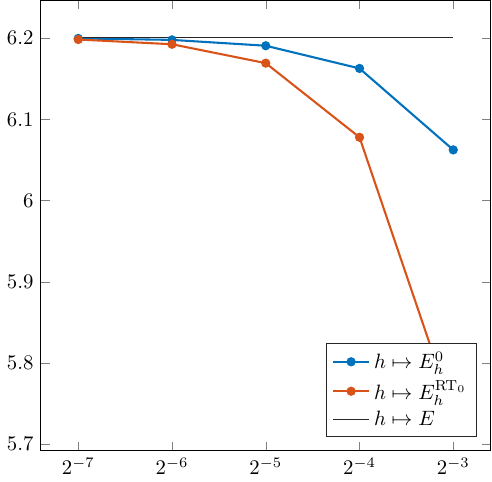}
	\end{minipage}\\[.25cm]
	\begin{minipage}[b]{.32\linewidth}
		\includegraphics[width=\linewidth]{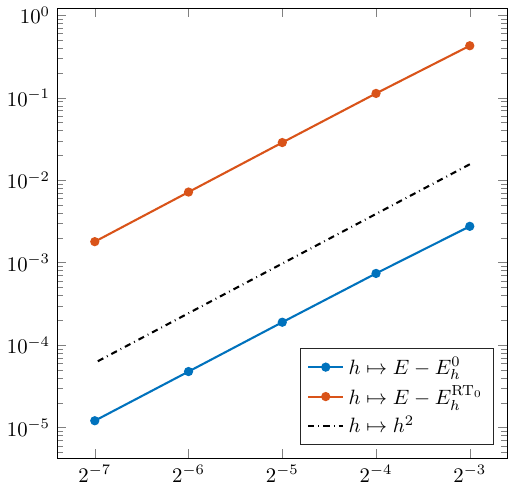}
	\end{minipage}\hfill
	\begin{minipage}[b]{.32\linewidth}
		\includegraphics[width=\linewidth]{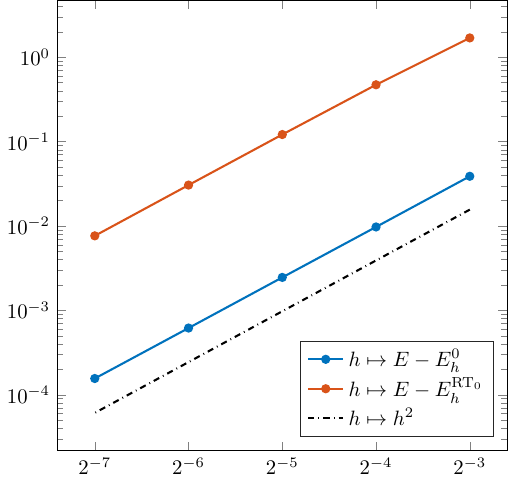}
	\end{minipage}\hfill
	\begin{minipage}[b]{.32\linewidth}
		\includegraphics[width=\linewidth]{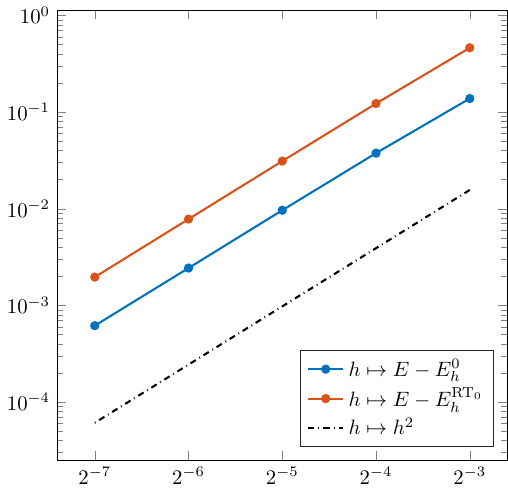}
	\end{minipage}
	\caption{First row: energy approximations \( E_h^0 \) of the modified HHO method and \( E_h^{\mathrm{RT}_0} \) of the post-processed Raviart--Thomas method (see~\cite{GHL24}) for the potentials \( V_1 \), \( V_2 \), and \( V_3 \) (from left to right). Second row: corresponding approximation errors.
	}
	\label{fig:lowbound}
\end{figure}
We consider the mesh hierarchy \( \{\mathcal{T}_h \with h = 2^{-3}, \dots, 2^{-7}\} \). To satisfy the assumption of a piecewise constant potential required in \cref{theo:lower_bound}, the potentials~\( V_1 \) and \( V_2 \) are projected onto the space of piecewise constant functions defined on a uniform Cartesian grid with mesh size \( h = 2^{-2} \). The potentials on all finer meshes are obtained via prolongation. 
The computations are carried out using the modified HHO method introduced in \cref{eq:hhoapprox0}, where the stabilisation parameter~\( \sigma \) is determined by rearranging \cref{eq:condstabilization} for \( \sigma \) and estimating \( E_h^0 \) from above using the reference energy~$E$, which is computed via the conforming $\mathcal Q^1$-finite element method (as outlined above). For the considered mesh sizes, the resulting stabilisation parameter is of order one, and it satisfies the condition in \cref{eq:condstabilization} by construction.

The first row of \cref{fig:lowbound} confirms that the energy approximation \( E_h^0 \) indeed provides guaranteed lower energy bounds, as theoretically predicted by \cref{theo:lower_bound}. For comparison, we also show the post-processed energy approximations obtained using the lowest-order Raviart--Thomas discretisation from \cite{GHL24}, which also yields guaranteed lower energy bounds.
The second row of \cref{fig:lowbound} displays the energy approximation errors for the modified HHO method and the post-processed Raviart--Thomas method. 
%
Notably, for the potentials \( V_1 \) and \( V_2 \), the modified HHO method yields significantly more accurate lower bounds (by approximately two orders of magnitude for \( V_1 \) and one and a half orders for \( V_2 \)) compared to the Raviart--Thomas method. The discrepancy arises from the post-processing step in the Raviart--Thomas method, which dominates the error in case of the smooth potentials \( V_1 \) and \( V_2 \), where the discretisation error is comparatively small.
For the rough potential \( V_3 \), the lower energy bounds provided by the modified HHO method are still more accurate than that of the post-processed Raviart--Thomas method, though the improvement is less pronounced (about a factor of three). 

\subsection{Optimal order convergence}
In the following, we investigate the convergence properties of the HHO method introduced in~\cref{eq:hhoapprox} and its modified version from~\cref{eq:hhoapprox0}.
For the corresponding numerical experiments, we consider the harmonic potential \( V_1 \) and set the particle interaction parameter to \( \kappa = 1000 \). The potential is integrated exactly using a quadrature rule of sufficiently high order.

\begin{figure}
	\begin{minipage}[b]{.32\linewidth}
		\includegraphics[width=\linewidth]{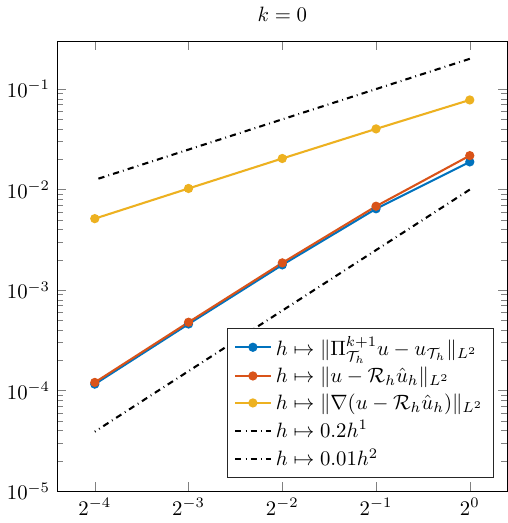}
	\end{minipage}\hfill
	\begin{minipage}[b]{.32\linewidth}
		\includegraphics[width=\linewidth]{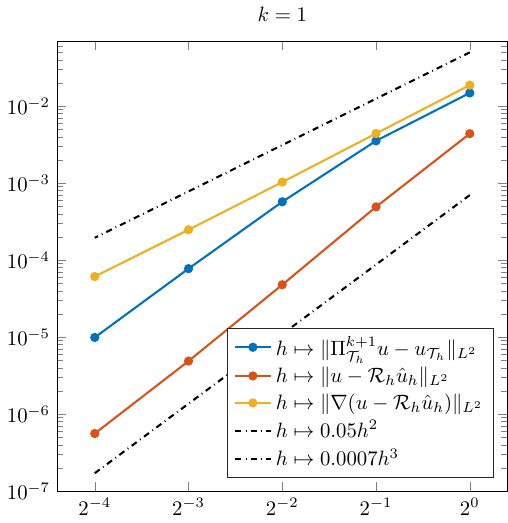}
	\end{minipage}\hfill
	\begin{minipage}[b]{.32\linewidth}
		\includegraphics[width=\linewidth]{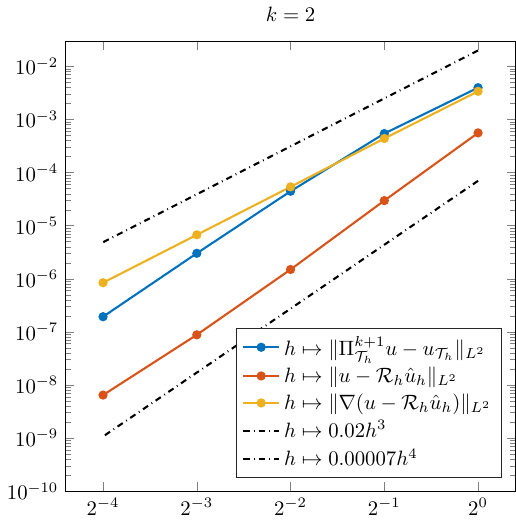}
	\end{minipage}\\[.25cm]
	\begin{minipage}[b]{.32\linewidth}
		\includegraphics[width=\linewidth]{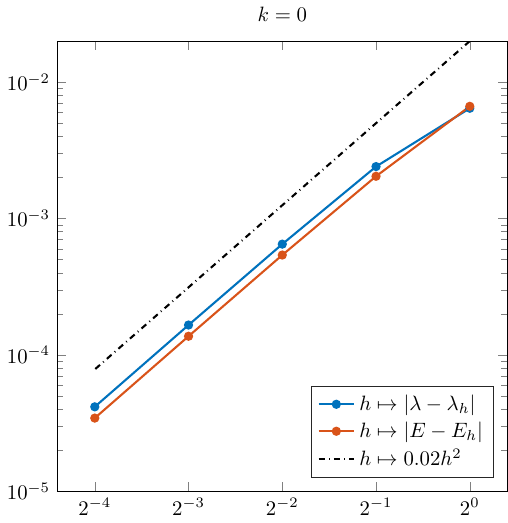}
	\end{minipage}\hfill
	\begin{minipage}[b]{.32\linewidth}
		\includegraphics[width=\linewidth]{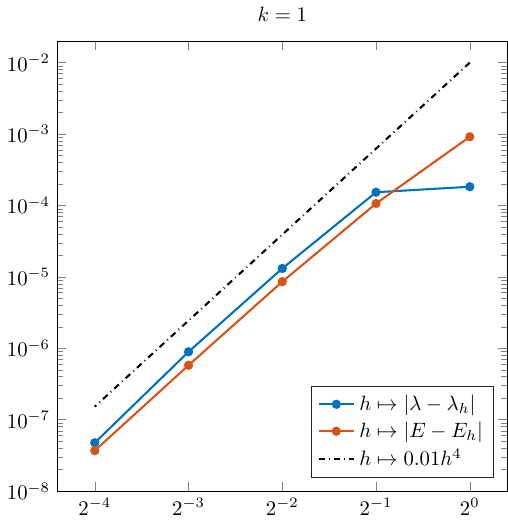}
	\end{minipage}\hfill
	\begin{minipage}[b]{.32\linewidth}
		\includegraphics[width=\linewidth]{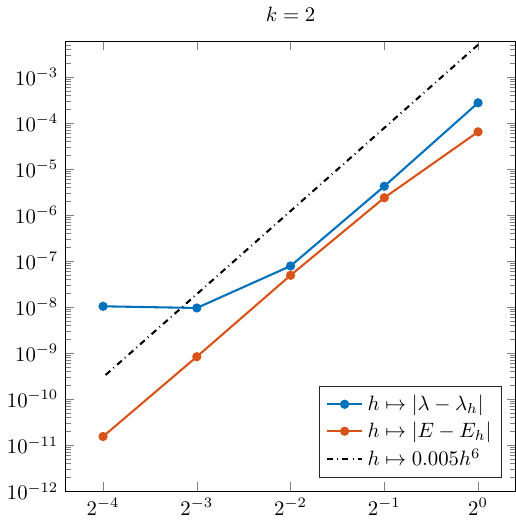}
	\end{minipage}
	\caption{First row: convergence plots of the (reconstructed) ground state approximations of the HHO method for polynomial degrees $k = 1,2,3$ (from left to right). Second row: convergence plots of the corresponding eigenvalue and energy approximations.}
	\label{fig:conv}
\end{figure}

To demonstrate the convergence of the HHO method introduced in \cref{eq:hhoapprox}, we employ the mesh hierarchy \( \{ \mathcal{T}_h \with h = 2^0, \dots, 2^{-4} \} \). In \cref{fig:conv} one observes that the HHO method exhibits optimal convergence rates, in agreement with the theoretical predictions of \cref{theo:err_est,theo:l2errorest}.
Note that, both \( \|u - u_{\mathcal{T}_h}\|_{L^2} \) and \( \|u - \mathcal{R}_h \hat{u}_h\|_{L^2} \) converge at the expected rate of \( \mathcal{O}(h^{k+2}) \); however, for polynomial degrees $k>0$, the latter error is consistently smaller by approximately one order of magnitude.
In the lower-left plot of \cref{fig:conv}, one observes that the eigenvalue approximation for the HHO method with polynomial degree \( k = 2 \) stagnates at an error level of approximately \( 10^{-8} \). This behavior is likely due to numerical effects such as finite machine precision and the stopping criteria of the nonlinear solver.

For the modified HHO method, we consider the mesh hierarchy \( \{ \mathcal{T}_h \with h = 2^0, \dots, 2^{-6} \} \). After an initial plateau in the error, which is due to the nonlinear solver for \cref{eq:hhoapprox0} not converging within the maximum number of iterations, the expected convergence rate predicted by \cref{theo:err_est0} is clearly observed.

\begin{figure}
	\includegraphics[width=.32\linewidth]{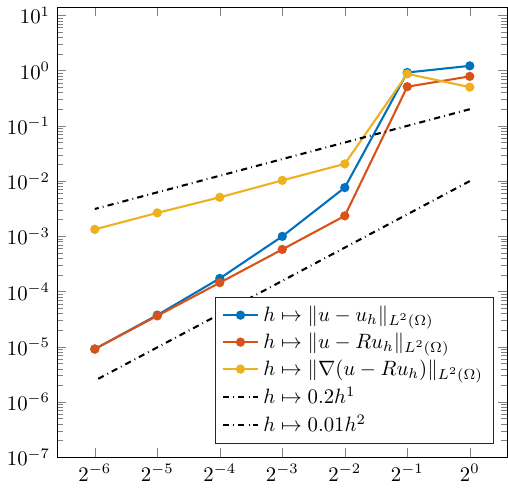}\hspace{.5cm}
	\includegraphics[width=.32\linewidth]{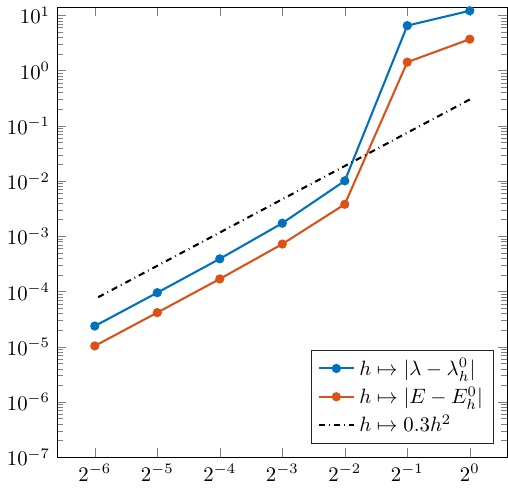}
	\caption{Left: convergence plot of the ground state approximation of the modified HHO method. Right: convergence plot of the corresponding eigenvalue and energy approximations (right).
	}
	\label{fig:convmod}
\end{figure}

\section{Conclusion}

In conclusion, we have demonstrated the effective application of a hybrid high-order (HHO) discretisation to the Gross--Pitaevskii eigenvalue problem. We have proved the optimal-order convergence both for the classical HHO ground state approximation and for a modified lowest-order variant that provides guaranteed lower bounds on the ground state energy. Notably, these bounds are obtained without any post-processing. Numerical experiments confirm that, particularly for smooth problems, the proposed method produces significantly more accurate guaranteed lower energy bounds than the post-processing-based approach~\cite{GHL24}.

\section*{Acknowledgments}
M. Hauck acknowledges funding from the Deutsche Forschungsgemeinschaft\linebreak  (DFG, German Research Foundation) -- Project-ID 258734477 -- SFB 1173. Y.~Liang was supported in part through the Royal Society University Research Fellowship (URF\textbackslash R1\textbackslash 221398, RF\textbackslash ERE\textbackslash 221047). Moreover, the authors would like to thank Ngoc Tien Tran for providing a basic implementation of the HHO method, and Andreas Rupp for fruitful discussions on the Gross--Pitaevskii problem.

\appendix
\section{Collection of frequently used bounds}

The first two results are a Poincaré inequality and a trace inequality, both stated with explicit constants.
\begin{lemma}[Poincare inequality]\label{lem:poincare}
		For all $T \in \mathcal T_h$ and any $v\in H^1(T)$ with $\int_T v\dx=0$, it holds that
	\begin{equation*}
		\|v\|_{L^2(T)} \leq \pi^{-1}h_T \|\nabla v\|_{L^2(T)}
	\end{equation*}
	with $\pi>0$ denoting the circle constant.
\end{lemma}
\begin{proof}
The proof can be found, for example, in~\cite{Bebendorf2003}.
\end{proof}
\begin{lemma}[Trace inequality]\label{lem:trace_theo}
	For all $T \in \mathcal T_h$ and any $v\in H^1(T)$ satisfying ${\int_T v\dx=0}$, it holds that
	\begin{equation*}
		\sum_{F \in \mathcal F_{\partial T}}\ell_{T,F}^{-1}\|u\|_{L^2({F})}^2 \leq C_{\mathrm{tr}} \|\nabla u\|_{L^2(T)}^2
	\end{equation*}
	with the constant $C_{\mathrm{tr}} = 1/\pi^2+2/(d\pi)>0$.
\end{lemma}
\begin{proof}
The result and its corresponding proof can be found in \cite[Lem.~3.1]{Tran2024}, which, in turn, is based on \cite[Lem.~7.2]{Gal23}.
\end{proof}

Next, we present several results relevant to the analysis of the HHO method.

\begin{lemma}[Discrete Sobolev embeddings]\label{prop:discrete:sobolev:embed}
	Let $d \in \{2,3\}$ and $q$ satisfy $1 \leq q < \infty$ if $d=2$, and $1 \leq q \leq 6$ if $d=3$. Then, for all $\hat{v}_h = (v_{\mathcal{T}_h}, v_{\mathcal{F}_h}) \in \hat{U}_h$, it holds~that
	\[
	\|v_{\mathcal{T}_h}\|_{L^q} \lesssim \|\hat{v}_h\|_{a_h}
	\]
	with hidden constant depending only on the domain, mesh regularity, $q$, $k$, and~$\sigma$.
\end{lemma}

\begin{proof}
	The proof of this result can be deduced by combining \cite[Prop.~5.4]{Pietro2017} and \cite[Lem.~2.18]{Pietro2020}.
\end{proof}
\begin{lemma}[Moment-preserving smoothing operator]\label{prop:prese_inter}
	There exists a linear operator $\mathcal{J}_h  \colon \allowbreak \hat{U}_h \to H_0^1(\Omega)$ satisfying, for all $\hat{v}_h \in \hat{U}_h$, that
	\begin{equation*}
		\mathcal I_h\circ \mathcal J_h\hat{v}_h = \hat{v}_h,
	\end{equation*}
	which implies the following orthogonality relations:
	\[
	v_{\mathcal{T}_h} - \mathcal{J}_h \hat{v}_h \perp_{L^2} \mathcal{P}^{k+1}(\mathcal{T}_h), 
	\quad 
	\nabla_h(\mathcal{J}_h \hat{v}_h - \mathcal{R}_h \hat{v}_h) \perp_{L^2} \nabla_h \mathcal{P}^{k+1}(\mathcal{T}_h).
	\]
Moreover, the operator satisfies, for all $ \hat v_h \in \hat U_h$, the stability estimate
	\begin{equation*}
		\|\mathcal J_h\hat{v}_h\|_{H^1}\lesssim \|\hat{v}_h\|_{a_h},
	\end{equation*}
	and, for any $v \in H^2(\Omega)\cap H^1_0(\Omega)$, the following approximation properties hold:
	\begin{equation*}
		\|\nabla(\mathcal J_h\mathcal I_hv-v)\|_{L^2}\lesssim h \|v\|_{H^2},\quad \|\mathcal J_h\mathcal I_hv-v\|_{L^2}\lesssim h^2 \|v\|_{H^2},
	\end{equation*}
	where hidden constants depend only on the domain, mesh regularity, $k$, and~$\sigma$.
\end{lemma}

\begin{proof}
	For the construction of such an operator and the corresponding analysis, we refer, for example, to 	\cite[Sec.~4.3]{ErnZanotti2020}. The proof of the stability and approximation properties uses similar arguments as \cite[Lem. 3.5]{LiangNgoc2025}.
\end{proof}

\begin{lemma}[Discrete $L^{\infty}$-bound]\label{lem:linf_bound}
	Assume that $V \in H^1(\mathcal T_h)$. Then, for any discrete ground state $\hat{u}_h = (u_{\mathcal{T}_h}, u_{\mathcal{F}_h})$ of \cref{eq:hhoapprox}, both $\|\mathcal{J}_h \hat{u}_h\|_{L^\infty}$ and $\|u_{\mathcal{T}_h}\|_{L^\infty}$ are uniformly bounded. 
	The same holds for the modified HHO approximation $\hat{u}_h^0 = (u_{\mathcal{T}_h}^0, u_{\mathcal{F}_h}^0)$ of \cref{eq:hhoapprox0} and its smoothed version $\mathcal{J}_h \hat{u}_h^0$.
\end{lemma}

\begin{proof}
	We proceed as in the proof of \cref{theo:plain_conver}, observing that \(\hat{u}_h\) can be interpreted as the HHO approximation of the solution \(u_h^c \in H^2(\Omega) \cap H_0^1(\Omega)\) to the auxiliary Poisson problem \cref{eq:ustar}. Recall that the \(H^2\)-norm of \(u_h^c\) is uniformly bounded.  
The improved \(L^2\)-error estimate for the HHO method, as established for example in \cite[Thm.~2.32 \&~2.33]{Pietro2020}, then implies that
	\begin{equation*}
		\|\mathcal{R}_h\hat{u}_h-u_h^c\|_{L^2} + \|\Pi_{\mathcal{T}_h}^{k+1}u_h^c - u_{\mathcal{T}_h}\|_{L^2}\lesssim h^2.
	\end{equation*}
Using a comparison between the discrete \( L^2 \)- and \( L^\infty \)-norms, the stability of the \( L^2 \)-projection, and the uniform \( L^\infty \)-bound for \( u_h^c \), we obtain for any \( T \in \mathcal{T}_h \) that
	\begin{align*}
		\|u_{\mathcal{T}_h}\|_{L^{\infty}(T)} &\leq \|\Pi_{\mathcal{T}_h}^{k+1}u_h^c - u_{\mathcal{T}_h}\|_{L^{\infty}(T)} + \|\Pi_{\mathcal{T}_h}^{k+1}u_h^c\|_{L^{\infty}(T)}\\
		&\lesssim h_T^{-d/2}\|\Pi_{\mathcal{T}_h}^{k+1}u_h^c - u_{\mathcal{T}_h}\|_{L^2(T)} + h_T^{-d/2}\|u_h^c\|_{L^{2}(T)}\\ &\lesssim h_T^{(4-d)/2} + \|u_h^c\|_{L^{\infty}(T)}\lesssim 1.
	\end{align*}
	Similarly arguments also show that  $\mathcal{R}_h\hat{u}_h$ is uniformly $L^{\infty}$-bounded.
	
Next, we show that $\mathcal J_h \hat{u}_h$ is uniformly $L^{\infty}$-bounded. Using again the comparison result between discrete norms, we obtain for any \(T \in \mathcal{T}_h\) that
	\begin{align*}
		\|\mathcal J_h\hat{u}_h\|_{L^{\infty}(T)} & \lesssim h_T^{-d/2} \|\mathcal J_h\hat{u}_h\|_{L^{2}(T)}\\
		&\lesssim h_T^{-d/2} \sum_{T'\cap T\neq \emptyset} \Big\{\|\mathcal{R}_h\hat{u}_h \|_{L^2(T')} + \|\mathcal{R}_h\hat{u}_h-u_{\mathcal{T}_h} \|_{L^2(T')} \\[-1.5ex]
		 & \ \ \ \ \ \ \ \qquad \quad \qquad \ \ +\sum_{F\in\mathcal{F}_{\partial T'}}h_F^{1/2}\|\Pi_{F}^k(\mathcal{R}_h\hat{u}_h-u_{\mathcal{F}_h})\|_{L^2(F)}\Big\}\\
		&\lesssim \|\mathcal{R}_h\hat{u}_h \|_{L^{\infty}} + \|\mathcal{R}_h\hat{u}_h-u_{\mathcal{T}_h} \|_{L^{\infty}} + h^{-(d-2)/2}|\hat{u}_h|_{s_h} \lesssim 1,
	\end{align*}
	where the second inequality follows from the specific construction of \( \mathcal{J}_h \) in \cite[Sec.~4.3]{ErnZanotti2020}, and arguments similar to those in \cite[Lem.~3.5]{LiangNgoc2025}. To derive the last inequality, we use that \( |\hat{u}_h|_{s_h} \lesssim h \|u_h^c\|_{H^2} \) (see \cref{theo:err_est} or \cite[Thm.~2.28]{Pietro2020}). This leads the uniform \( L^{\infty} \)-bound for \( \mathcal{J}_h \hat{u}_h \). The corresponding proof for the modified HHO approximation \( \hat{u}_h^0 \) is analogous and is therefore omitted for brevity.
\end{proof}

\begin{proposition}\label{eq:JhIh_linf}
For any \(v \in H^2(\Omega) \cap H_0^1(\Omega)\), it holds that
	\begin{align*}
		\|\mathcal J_h\mathcal{I}_hv\|_{L^{\infty}} 
		\lesssim \|\mathcal{R}_h\mathcal{I}_hv \|_{L^{\infty}} + \|\mathcal{R}_h\mathcal{I}_hv-\Pi_{\mathcal{T}_h}^{k+1}v \|_{L^{\infty}} + h^{-(d-2)/2}|\mathcal{I}_hv|_{s_h} \lesssim \|v\|_{H^2}.
	\end{align*}
\end{proposition}

\begin{lemma}[Quadrature error]\label{lem:modi_1}
For any \(\hat{v}_h = (v_{\mathcal{T}_h}, v_{\mathcal{F}_h})\), \(\hat{w}_h = (w_{\mathcal{T}_h}, w_{\mathcal{F}_h}) \in \hat{U}_h\), we have the following two estimates:
	\begin{align}\label{eq:quaderror1}
		\begin{split}
				&|((\Pi_{\mathcal{T}_h}^0{v}_{\mathcal{T}_h})^2 -v_{\mathcal{T}_h}^2 , v_{\mathcal{T}_h}w_{\mathcal{T}_h})_{L^2}|\\&\; \lesssim  h^2\|v_{\mathcal{T}_h}\|_{L^{\infty}}\|\nabla \mathcal J_h\hat{v}_h\|_{L^2} \big(\|v_{\mathcal{T}_h}\|_{L^{\infty}}\|\nabla \mathcal J_h\hat{w}_h\|_{L^2} + \|\nabla \mathcal J_h\hat{v}_h\|_{L^{2}}\|\mathcal J_h\hat{w}_h\|_{L^{\infty}}\\ & \qquad \qquad \quad  \quad  \qquad  \quad  \ \ \ \ + \|\mathcal J_h\hat{v}_h\|_{L^{\infty}}\|\nabla \mathcal J_h\hat{w}_h\|_{L^2}+ \|w_{\mathcal{T}_h}\|_{L^{\infty}}\|\nabla \mathcal J_h\hat{v}_h\|_{L^2}\big),
		\end{split}
	\end{align}\vspace{-.5ex}
		\begin{align}\label{eq:quaderror2}
			\begin{split}
					&| ( v_{\mathcal{T}_h}^2, v_{\mathcal{T}_h}w_{\mathcal{T}_h}- \Pi_{\mathcal{T}_h}^0v_{\mathcal{T}_h} \Pi_{\mathcal{T}_h}^0w_{\mathcal{T}_h})_{L^2}|\\
				&\quad \lesssim h^2\big(\|v_{\mathcal{T}_h}\|_{L^{\infty}}\|\nabla \mathcal J_h\hat{v}_h\|_{L^2}+\|\mathcal J_h\hat{v}_h\|_{L^{\infty}}\|\nabla \mathcal J_h\hat{v}_h\|_{L^2} \big)\\
				&\qquad \quad \qquad \times \big(\|v_{\mathcal{T}_h}\|_{L^{\infty}}\|\nabla \mathcal J_h\hat{w}_h\|_{L^2} 
				+ \|w_{\mathcal{T}_h}\|_{L^{\infty}}\|\nabla \mathcal J_h\hat{v}_h\|_{L^2}\big),
			\end{split}
	\end{align}
	where hidden constants depend only on the domain, mesh regularity, $k$, and~$\sigma$.
\end{lemma}
\begin{proof}
	We begin the proof by establishing three auxiliary estimates that will be used later. The first one, which holds for any \( \hat{v}_h = (v_{\mathcal{T}_h}, v_{\mathcal{F}_h}) \in \hat{U}_h \), reads
	\begin{align}\label{eq:appendix:auxi_1}
		\|v_{\mathcal{T}_h} - \Pi_{\mathcal{T}_h}^0 v_{\mathcal{T}_h}\|_{L^2} 
		&\leq \|v_{\mathcal{T}_h} - \mathcal{J}_h \hat{v}_h\|_{L^2} + \|\mathcal{J}_h \hat{v}_h - \Pi_{\mathcal{T}_h}^0 v_{\mathcal{T}_h}\|_{L^2} \lesssim h \|\nabla \mathcal{J}_h \hat{v}_h\|_{L^2}.
	\end{align}
This follows from \( \Pi_{\mathcal{T}_h}^{k+1} \mathcal{J}_h \hat{v}_h = v_{\mathcal{T}_h} \) and the approximation properties of the $L^2$-projection.
To derive the second estimate, we apply the triangle inequality along with the first estimate, which gives, for any \( \hat{v}_h, \hat{w}_h \in \hat{U}_h \), that
	\begin{equation}\label{eq:appendix:auxi_2}
		\begin{aligned}
			\|v_{\mathcal{T}_h}w_{\mathcal{T}_h}-\mathcal J_h\hat{v}_h\mathcal J_h\hat{w}_h\|_{L^2}&\leq \|v_{\mathcal{T}_h}w_{\mathcal{T}_h}-v_{\mathcal{T}_h} \mathcal J_h\hat{w}_h\|_{L^2} + \|v_{\mathcal{T}_h} \mathcal J_h\hat{w}_h-\mathcal J_h\hat{v}_h\mathcal J_h\hat{w}_h\|_{L^2}\\
			&\lesssim h\|v_{\mathcal{T}_h}\|_{L^{\infty}}\|\nabla\mathcal J_h\hat{w}_h\|_{L^2} + h\|\mathcal J_h\hat{w}_h\|_{L^{\infty}}\|\nabla\mathcal J_h\hat{v}_h\|_{L^2}.
		\end{aligned}
	\end{equation}
To prove the third estimate, we use the triangle inequality, the approximation properties of the \( L^2 \)-projection, and the product rule, which yields that
	\begin{equation}\label{eq:appendix:auxi_3}
		\begin{aligned}
			&\|v_{\mathcal{T}_h}w_{\mathcal{T}_h}- \Pi_{\mathcal{T}_h}^0(v_{\mathcal{T}_h}w_{\mathcal{T}_h})\|_{L^2}\\
			 &\quad \leq \|v_{\mathcal{T}_h}w_{\mathcal{T}_h}-\mathcal J_h\hat{v}_h\mathcal J_h\hat{w}_h\|_{L^2}  + \|\mathcal J_h\hat{v}_h\mathcal J_h\hat{w}_h -\Pi_{\mathcal{T}_h}^0(\mathcal J_h\hat{v}_h\mathcal J_h\hat{w}_h)\|_{L^2}\\ & \ \ \ \ \ \quad \quad   + \|\Pi_{\mathcal{T}_h}^0(v_{\mathcal{T}_h}w_{\mathcal{T}_h})-\Pi_{\mathcal{T}_h}^0(\mathcal J_h\hat{v}_h\mathcal J_h\hat{w}_h)\|_{L^2}\\
			&\quad \lesssim h\big(\|v_{\mathcal{T}_h}\|_{L^{\infty}}\|\nabla \mathcal J_h\hat{w}_h\|_{L^2} + \|\nabla \mathcal J_h\hat{v}_h\|_{L^{2}}\|\mathcal J_h\hat{w}_h\|_{L^{\infty}}
			+ \|\mathcal J_h\hat{v}_h\|_{L^{\infty}}\|\nabla \mathcal J_h\hat{w}_h\|_{L^2}\big).
		\end{aligned}
	\end{equation}
We now have the tools to prove \cref{eq:quaderror1}. Applying the triangle inequality yields
	\begin{multline*}
		|((\Pi_{\mathcal{T}_h}^0{v}_{\mathcal{T}_h})^2 -v_{\mathcal{T}_h}^2 , v_{\mathcal{T}_h}w_{\mathcal{T}_h})_{L^2}|  \leq \underbrace{|((\Pi_{\mathcal{T}_h}^0{v}_{\mathcal{T}_h})^2 -v_{\mathcal{T}_h}^2 , \Pi_{\mathcal{T}_h}^0(v_{\mathcal{T}_h}w_{\mathcal{T}_h}))_{L^2}|}_{\eqqcolon \Xi_1}\\+ 
		\underbrace{|((\Pi_{\mathcal{T}_h}^0{v}_{\mathcal{T}_h})^2 -v_{\mathcal{T}_h}^2 , v_{\mathcal{T}_h}w_{\mathcal{T}_h}-\Pi_{\mathcal{T}_h}^0(v_{\mathcal{T}_h}w_{\mathcal{T}_h}) )_{L^2}|}_{\eqqcolon \Xi_2},
	\end{multline*}
	where the term $\Xi_1$ can be estimated as
	\begin{align}\label{eq:xi1}
		\Xi_1 & \leq \sum_{T \in \mathcal{T}_h}|\Pi_{T}^0(v_{T}w_{T})|\int_{T}(v_T-\Pi_{T}^0{v}_{T})^2\mathrm{d}x
		 \lesssim h^2\|v_{\mathcal{T}_h}\|_{L^{\infty}}\|w_{\mathcal{T}_h}\|_{L^{\infty}}\|\nabla \mathcal J_h\hat{v}_h\|_{L^2}^2.
	\end{align}
Here, the first inequality follows from elementary algebraic manipulations, using \( \Pi_T^0 v = |T|^{-1} \int_T v \, \mathrm{d}x \), while the second follows from \( \|\Pi_T^0 v\|_{L^\infty(T)} \leq \|v\|_{L^\infty(T)} \) and~\cref{eq:appendix:auxi_1}. The term $\Xi_2$ can be estimated as
	\begin{align*}
		\Xi_2 &\lesssim h\|v_{\mathcal{T}_h}\|_{L^{\infty}}\|\nabla \mathcal J_h\hat{v}_h\|_{L^2}\|v_{\mathcal{T}_h}w_{\mathcal{T}_h}- \Pi_{\mathcal{T}_h}^0(v_{\mathcal{T}_h}w_{\mathcal{T}_h})\|_{L^2},
	\end{align*}
	where we again use the $L^{\infty}$-bound of $\Pi_T^0$ and \cref{eq:appendix:auxi_1}. Applying \cref{eq:appendix:auxi_3} to estimate the last term on the right-hand side yields the desired bound for \(\Xi_2\). Combining the previous estimates then establishes \cref{eq:quaderror1}.
	
Next, we prove estimate \cref{eq:quaderror2}. Applying the triangle inequality, we obtain that
\begin{multline*}
	| ( v_{\mathcal{T}_h}^2, v_{\mathcal{T}_h}w_{\mathcal{T}_h}- \Pi_{\mathcal{T}_h}^0v_{\mathcal{T}_h} \Pi_{\mathcal{T}_h}^0w_{\mathcal{T}_h})_{L^2}|  \leq 
	\underbrace{| ( \Pi_{\mathcal{T}_h}^0(v_{\mathcal{T}_h}^2), v_{\mathcal{T}_h}w_{\mathcal{T}_h}- \Pi_{\mathcal{T}_h}^0v_{\mathcal{T}_h} \Pi_{\mathcal{T}_h}^0w_{\mathcal{T}_h})_{L^2}|}_{\eqqcolon \xi_1}\\
	+ \underbrace{| ( v_{\mathcal{T}_h}^2 - \Pi_{\mathcal{T}_h}^0(v_{\mathcal{T}_h}^2), v_{\mathcal{T}_h}w_{\mathcal{T}_h}- \Pi_{\mathcal{T}_h}^0v_{\mathcal{T}_h} \Pi_{\mathcal{T}_h}^0w_{\mathcal{T}_h})_{L^2}|}_{\eqqcolon \xi_2}.
\end{multline*}
To estimate the term \(\Xi_1\), we apply algebraic manipulations similar to those in the proof of \cref{eq:xi1}, as well as the \(L^\infty\)-bound of \(\Pi_T^0\) and \cref{eq:appendix:auxi_1}, which gives
\begin{align*}
	\xi_1 &\leq 
	\sum_{T \in \mathcal{T}_h}|\Pi_{T}^0(v_{T}^2)||\int_{T}(v_{T}- \Pi_{T}^0v_{T})(w_{T}- \Pi_{T}^0w_{T}) \mathrm{d}x|\\
	&\lesssim h^2\|v_{\mathcal{T}_h}\|_{L^{\infty}}^2\|\nabla \mathcal J_h\hat{v}_h\|_{L^2}\|\nabla \mathcal J_h\hat{w}_h\|_{L^2}.
\end{align*}
The term \(\xi_2\) can be estimated by applying \cref{eq:appendix:auxi_3} with \(\hat{v}_h = \hat{w}_h\), proceeding similarly to the proof of \cref{eq:appendix:auxi_2}. Combining these estimates yields \cref{eq:quaderror2}.
\end{proof}

	\bibliographystyle{amsalpha}
	\bibliography{bib}
\end{document}